\documentclass[10pt]{article}
\usepackage{fullpage,algorithm,algpseudocode,latexsym,amsmath,amsfonts,amscd,amsthm,graphicx,epstopdf,lineno,subfig,color}
\usepackage[colorlinks=true,citecolor=blue,urlcolor=blue]{hyperref}
\usepackage[round-precision=2,round-mode=figures,scientific-notation=true]{siunitx}
\newtheorem{theorem}{Theorem}[]

\newtheorem{remark}{Remark}
\usepackage{mcode}
\usepackage{authblk}

\date{}

\title{The Gamma Function via Interpolation}
\author{Matthew F. Causley \\ {\href{mailto:mcausley@kettering.edu}{mcausley@kettering.edu}}}
\affil{Department of Mathematics, Kettering University}

\begin{document}
	\maketitle
	
\begin{abstract}
	A new computational framework for evaluation of the gamma function $\Gamma(z)$ over the complex plane is developed. The algorithm is based on interpolation by rational functions, and generalizes the classical methods of Lanczos \cite{Lanczos} and Spouge \cite{Spouge} (which we show are also interpolatory). This framework utilizes the exact poles of the gamma function. By relaxing this condition and allowing the poles to vary, a near-optimal rational approximation is possible, which is demonstrated using the adaptive Antoulous Anderson (AAA) algorithm, developed in \cite{AAA,AAA_2020}. The resulting approximations are competitive with Stirling's formula in terms of overall efficiency.

\end{abstract}
	
\section{Introduction}
In the fall of 1729, Christian Goldbach exchanged a series of letters across Europe with Leonhard Euler and Daniel Bernoulli, discussing the sequence $1, 2, 6, 24, \ldots$, which we recognize as the factorials $a_n = n!$. They sought to interpolate the index, so that $x!$ would be consistently defined over the positive reals \cite{Dutka}. Bernoulli discovered an infinite product representation \cite{Gronau}, which Euler soon refined into an interpolating formula. He then pursued a definition based on integrals, eventually producing the gamma function in its modern form
\begin{equation}
	\label{eqn:Gamma_Def}
	\Gamma(z) = \int_0^\infty t^{z-1}e^{-t}\, dt, \qquad \Re(z)>0.
\end{equation}
The notation is due to Legendre\footnote{Despite the mismatch in arguments, Legendre's notation $\Gamma(z)$ has prevailed over Gauss' notation $\Pi(z) =\Gamma(z+1)$, which satisfies $\Pi(n) = n!$. Although in certain cases \cite{Lanczos,Spouge} the notation $z! = \Gamma(z+1)$ is instead adopted, to avoid confusion.}, who described \eqref{eqn:Gamma_Def} as the Eulerian integral of the second kind \cite{Davis}\footnote{The integral of the first kind defines the beta function $B(\alpha,\beta) = \Gamma(\alpha)\Gamma(\beta)/\Gamma(\alpha+\beta)$.}. Often this integral is taken to be the standard definition (see e.g. \cite{Artin,Remmert} for a detailed discussion about properties and definitions of the gamma function).

Since its advent, the gamma function has come to play a vital role in nearly every branch of pure and applied mathematics, statistics, physics, chemistry and engineering. It is perhaps the most special of the special functions. The literature is vast, but fortunately reviews \cite{Davis,Borwein}, monographs \cite{Artin,Sebah,Dutka}, book chapters \cite{AAR,Whittaker,Abramowitz,BatemanV1}, and bibliographies \cite{Perez} greatly aid in its perusal.

Incidentally it was in the same decade as Euler and his contemporaries had set about their work, that James Stirling and Abraham DeMoivre conducted an investigation into central binomial coefficients and the natural log of factorials, which gave rise to the celebrated Stirling series\footnote{in fact, this series is due to De Moivre. The series obtained by Stirling is
	\[
		\ln z! \sim \ln(\sqrt{2\pi})+ Z\ln Z-Z+\sum_{k\geq 1} \frac{B_{2k}\left(\frac{1}{2}\right)}{2k(2k-1)}
	\frac{1}{Z^{2k-1}}, \quad Z = z+\frac{1}{2},
	\]
	which is slightly more accurate. See \cite{Dutka,CS}.}
\begin{equation}
	\label{eqn:DeMoivre}
	\ln \Gamma(z) \sim \ln \sqrt{2\pi} - z + \left(z-\frac{1}{2}\right)\ln z + \sum_{k\geq 1} \frac{B_{2k}}{2k(2k-1)}\frac{1}{z^{2k-1}}.
\end{equation}
This sum is given in terms of the Bernoulli numbers\footnote{Named after Jacob Bernoulli; Daniel was his nephew.} $B_{2k}=(-1)^{k+1} 2\zeta(2k)(2k)!/(2\pi)^{2k}$, which grow quickly with increasing $k$ and render the series divergent for finite $z$.

From the standpoint of numerical computation, many techniques have been considered \cite{Lanczos, Spouge, Luke, Luke2, Dubuc, Trefethen2006, Trefethen2007,AAA, Rump, Cardoso2019, Reinartz, Smith, Smith2} (we aim to provide merely a representative, but by no means comprehensive view of the literature). The approaches can be categorically sorted into:  i) asymptotic expansions ii) rational functions, and iii) numerical quadrature. Typically asymptotic expansions involve some variation of Stirling's series \eqref{eqn:DeMoivre}, which have been investigated quite thoroughly in the analytic number theory community (see e.g. \cite{Boyd,Mortici, LiChen, Nemes2015, Wang}, and references therein). Quadrature methods are usually applied directly to either Euler's definition \eqref{eqn:Gamma_Def} \cite{Reinartz, Dubuc}, or to the Hankel form for the reciprocal gamma function \cite{Trefethen2006, Trefethen2007}.

When high precision is required, the asymptotic series \eqref{eqn:DeMoivre} is still the gold standard (see, e.g. the excellent paper by Hare \cite{Hare}). Briefly, upon truncation at say $k = K$, and with a translation $z\to z+N$, the recursion
\begin{equation}
	\label{eqn:rec}
	\Gamma(z+1) = z\Gamma(z) \qquad \implies \qquad \ln \Gamma(z+1) = \ln z + \ln \Gamma(z),
\end{equation}
produces the following representation, suitable for computation
\begin{align}
	\nonumber
	\Gamma(z)	&\approx \frac{\sqrt{2\pi}(z+N)^{N+z-1/2}  }{z(z+1)(\ldots)(z+N-1)}\exp\left[-(z+N)+\sum_{k= 1}^K \frac{B_{2k}}{2k(2k-1)}\frac{1}{(z+N)^{2k-1}}\right] \\
	\label{eqn:SNK}
				&= \left[\sqrt{2\pi}(z+N)^{z-\frac{1}{2}}e^{-(z+N)}\right]
				 \left[\frac{(z+N)^N}{z(\ldots)(z+N-1)}\right] \exp\left[\sum_{k= 1}^K \frac{B_{2k}}{2k(2k-1)}\frac{1}{(z+N)^{2k-1}}\right].
\end{align}
The asymptotic expression \eqref{eqn:SNK} has been rearranged, sorted by behavior. The first bracketed term contains the "fast", or dominant asymptotic factors for large $z$. The second term does not contribute to the large $z$ asymptotics, but preserves the structure of $\Gamma(z)$; for $\Re(z)>-N$, this term incorporates the first $N$ poles (and a spurious zero of multiplicity $N$ at $z=-N$). The third bracketed contains the "slow" asymptotic factors, which correct the behavior as $z\to \infty$, and increase the rate of convergence there.

The convergence is optimal for large $|z|$ with $|\arg z|< \pi - \Delta$ for $\Delta >0$. The translation and truncation can also be chosen carefully to ensure fixed precision for small and moderate values \cite{Nemes2010,Boyd}, although it can be argued that this is less efficient than an approach which readily produces a fixed relative precision. This is a known property of the Lanczos approximation \cite{Lanczos}, as well as the related Spouge approximation \cite{Spouge}.

These latter methods can be motivated as follows. If we move the first bracketed term to the left, and generalize $N \to r \in \mathbb{R}$, we have a scaled version of the gamma function
\begin{equation}
	\label{eqn:Fr}
	F_r(z) = \frac{\Gamma(z)e^{z+r}}{(z+r)^{z-\frac{1}{2}}},
\end{equation}
with arbitrary parameter $r>0$. Independent of $r$, we observe that
\[
	\lim_{z\to \infty} F_r(z) = \sqrt{2\pi},
\]
prompting the rational approximation, which we first write as a sum of poles
\[
	F_r(z) \approx c_\infty(r)+\sum_{n=0}^{N-1} \frac{c_n(r)}{z+n}, \qquad r>N.
\]
so as to retain the first $N$ poles of the gamma function. Since $F_r(z)$ inherits these poles and additionally introduces a branch cut emanating from the singularity at $z=-r$, it is meromorphic, and single-valued in the right half plane $\Re(z+r)>0$. We therefore find
\begin{equation}
\label{eqn:FrN}
	\Gamma(z) \approx (z+r)^{z-\frac{1}{2}}e^{-z-r}\left[c_\infty(r)+\sum_{n=0}^{N-1} \frac{c_n(r)}{z+n}\right].
\end{equation}

The approximations due to Lanczos \cite{Lanczos}, and Spouge \cite{Spouge} both utilize \eqref{eqn:FrN} for computation, albeit with different motivation. The derivation presented by Lanczos is brilliant and elegant, relying on a Chebyshev series expansion of the integral \eqref{eqn:Gamma_Def} (after several changes of variables). The approximation can be put into the form \eqref{eqn:FrN}, and the determination of the residues $c_n(r)$ via Chebyshev coefficients ensures high precision for small and moderate values of $z$ over the right-half of the complex plane. By contrast Spouge directly relies on rigorous arguments of complex analysis, and in so doing sets $c_n(r)$ to be the residues of $F_r(z)$. While more concise, the Spouge approximation delivers slightly less accuracy (see section \ref{sec:Exist} for further discussion).

In the spirit of Goldbach, Euler and Bernoulli, we will obtain a novel formula of approximating $\Gamma(z)$ by appealing directly to interpolation theory. Suppose that $F_r(z)$ as defined in \eqref{eqn:Fr} is sampled at the (possibly complex) points $z_j$. Then we have the nearly-Cauchy system $C\textbf{c} = \textbf{f}$, where
\begin{equation}
	\label{eqn:Hf}
	C = \begin{bmatrix}
	1&\frac{1}{z_1} & \ldots &\frac{1}{z_1+N} \\
	\vdots &&& \vdots \\
	1&\frac{1}{z_M} & \ldots &\frac{1}{z_M+N}
	\end{bmatrix},
	\qquad \textbf{c} = \begin{bmatrix} c_\infty(r) \\c_0(r) \\ \vdots \\ c_{N-1}(r)	\end{bmatrix}.
	\qquad \textbf{f} = \begin{bmatrix} F_r(z_1) \\ F_r(z_2) \\ \vdots \\ F_r(z_N) \end{bmatrix}.
\end{equation}
Note that the system is nearly Cauchy, but for the first column of $C$, due to inclusion of the coefficient $c_\infty(r)$. The motivation of this paper began with an observation about this latter approach. In the special case that the interpolation points are chosen to be the first $N$ positive integers $z_j = j$, then the square matrix becomes a nearly-Hilbert matrix, and the method recovers precisely the Lanczos approximation \cite{Lanczos}. It is not clear whether Lanczos realized that his method preserved the property $\Gamma(n+1)=n!$ for $0\leq n \leq N$; and this fact does not appear elsewhere\footnote{at least not explicitly. But this property is used indirectly in \cite{Pugh} and \cite{Luke2} to compute expansion coefficients.}.

Alternatively we can reformulate \eqref{eqn:FrN} as
\begin{equation}
	\label{eqn:R}
	\Gamma(z)  \approx (z+r)^{z-\frac{1}{2}}e^{-z-r}R(z), \qquad R(z) = \frac{p_N(z)}{q_N(z)}, 
\end{equation}
where $p_N$ and $q_N$ are polynomials of degree $N$. Note that there is only a strict equivalence to the sum of poles expansion \eqref{eqn:FrN}, if we demand $q_N(z) = (z)^{(N)} = z(z+1)(\ldots)(z+N-1)$. Heuristically we argue that by allowing the poles to vary, we should expect greater accuracy elsewhere, particularly in the right half plane. Fitting to data then leads to the minimax problem
\[
	\min_{N} ||F_r(z) - R(z)||,
\]
where the norm can be taken over a discrete sampling set, or continuously over a subset $D$ of the complex domain $\mathbb{C}$. Below we will review the adaptive Anatoulous Anderson (AAA) algorithm proposed by Nakatsukasa, Set\'{e} and Trefethen \cite{AAA,AAA_2020}, although other methods such as RKFIT \cite{Berljafa}, vector fitting \cite{Gustavsen}, and bootstrap methods \cite{Xu} are equally applicable. Following the AAA algorithm, the rational function is obtained in barycentric form,
\begin{equation}
	\label{eqn:BR}
	R(z) = \frac{\mu(z)}{\nu(z)}, \qquad \mu(z) =\sum_{n=1}^N \frac{w_j f_j}{z-s_j}, \qquad \nu(z) = \sum_{n=1}^N \frac{w_j}{z-s_j},
\end{equation}
where $s_j$ are interpolation points in the complex plane, and $f_j = F_r(s_j)$ in the current setting. It should be noted that $s_j$ indicates neither the poles nor the zeros of $R(z)$, which is a fascinating aspect of the barycentric form.

The weights are obtained adaptively, by using a greedy algorithm to select from a discrete set of sample points which are simultaneously used to measure the norm of the error. In particular, $\textbf{w}$ is selected as the right singular vector of a (Loewner) divided difference matrix, corresponding to its smallest singular value. This simultaneously ensures accuracy, and numerical stability.

The AAA algorithm, as well as other constructions of rational approximations, fit the poles of a given function $f(z)$ numerically. But one uniquely attractive feature of this approach, is that the algorithm finds the poles directly from the sample set $s_j$, i.e. with no initial guess for the poles.

Hence our goal is twofold. First, we will establish a framework for the fixed-pole algorithm, which rely on solving a nearly-Cauchy system to select the strength of each pole $c_n(r)$ in \eqref{eqn:FrN}. This class contains as special cases the results of Lanczos (with $z_j = j$) and Spouge (which chooses the exact values of the residues, and hence interpolates $\Gamma(z)$ at $z = 0, -1, \ldots, -N+1$).

Secondly, we will construct a rational approximation to $F_r(z)$ as in \eqref{eqn:Fr}, in which the poles are allowed to vary, and thus are determined to produce high precision near the points $s_j$. As we will show below this slight modification leads to a highly accurate approximation that is also computationally expedient.

In general, the interpolation points can be chosen with various strategies, and may be complex. We will show below that the best results follow from laying the points either over the positive real line, or the line of symmetry $z = \frac{1}{2}+iy$. This is due to the fact that Euler's reflection formula
\begin{equation}
	\label{eqn:Ref}
	\Gamma(z) = \frac{\pi}{\sin(\pi z)\Gamma(1-z)},
\end{equation}
which analytically continues $\Gamma(z)$ into the left half of the complex plane. Therefore, the computational domain is restricted to $\Re(z)\geq \frac{1}{2}$. By the maximum modulus principle \cite{CKP}, the maximum error occurs along this line, and hence optimizing the accuracy there will ensure high precision over the entire complex plane (precision along the negative real axis is yet another matter, which we do not address here; see \cite{Rump} for further discussion).

The rest of the paper is laid out as follows. In the next section, we will review the relevant details for the existing approximations due to Lanczos and Spouge. We then proceed to generalize these formulas, and produce an interpolatory algorithm for $\Gamma(z)$ with the poles fixed. Several strategies for choosing points of interpolation are then considered. We then present results based on the general rational approximation using the AAA algorithm, and a brief comparison is made to the Stirling series. If the measure of efficiency is evaluation time to maintain a fixed relative precision $\epsilon$, the AAA-based algorithm is slightly more efficient that the Stirling series, which will inherently need to destroy extra digits of accuracy for larger $|z|$ to accommodate the precision for small $|z|$. We will then make some concluding remarks.

\section{Existing Approaches}
\label{sec:Exist}
The Lanczos approximation has been analyzed extensively; see for example \cite{Luke,Luke2,Pugh,Godfrey}. Other methods, such as those of Spouge \cite{Spouge} and those based on quadrature \cite{Trefethen2006,Trefethen2007,Dubuc} have also been explored. More recently, Chebyshev expansions have been applied directly to the Stirling series \cite{Reinartz}, which results in an expansion fo the form \eqref{eqn:F_N}; we will not pursue this any further.

The approximations of Lanczos and Spouge both utilize the asymptotically balanced gamma function \eqref{eqn:Fr}, with real parameter $r>0$. This function $F_r(z)$ is meromorphic for $\Re(z+r)>0$, with poles at the non-positive integers inherited from $\Gamma(z)$. A branch point at $z=-r$ is introduced by this approximation, and so the region of analyticity is over the complex plane away from the negative real axis.

\subsection{The Approximation of Spouge}
The approach taken by Spouge is concise, but only moderately accurate. The poles of $\Gamma(z)$ are simple, and their residues are readily computed by making use of the reflection formula \eqref{eqn:Ref}
\begin{equation}
	\label{eqn:residue}
	\lim_{z \to -k} (z+k)\Gamma(z) = \frac{1}{\Gamma(k+1)}\lim_{z\to -k} \frac{\pi(z+k)}{\sin(\pi z)} = \frac{(-1)^k}{k!}, \qquad k = 0, 1, \ldots.
\end{equation}
It follows that
\[
	c_n(r) = \frac{(-1)^n}{n!} e^{r-n}(r-n)^{n+1/2}, \qquad n = 0, 1, \ldots, N-1.
\]
Spouge sets $c_\infty(r)=\sqrt{2\pi}$, which ensures convergence as $z\to \infty$. We then have
\begin{equation}
\label{eqn:Spouge}
\Gamma(z) \approx \Gamma_S(z)=(z+r)^{z-\frac{1}{2}}e^{-(z+r)}\left[\sqrt{2\pi} +\sum_{n=0}^{N-1} \frac{(-1)^{n}e^{r-n}(r-n)^{n+\frac{1}{2}}}{n!(z+n)}\right].
\end{equation}
An extensive analysis of the parameter $r$ was conducted by Pugh \cite{Pugh} for the Lanczos approximation. The analysis is quite general, and can also be applied to the Spouge approximation. We can choose $r>N-1$ to approximate an additional point $\bar{z}$, and so $r = r(\bar{z})$ which satisfies the nonlinear equation $\Gamma(\bar{z})=\Gamma_S(\bar{z})$. We display $r(\bar{z})$ for several values of $N$ and $\bar{z}$ in Table \ref{tab:Spouge}.
\begin{figure}[t!]
	\centering
	\includegraphics[width=0.48\linewidth]{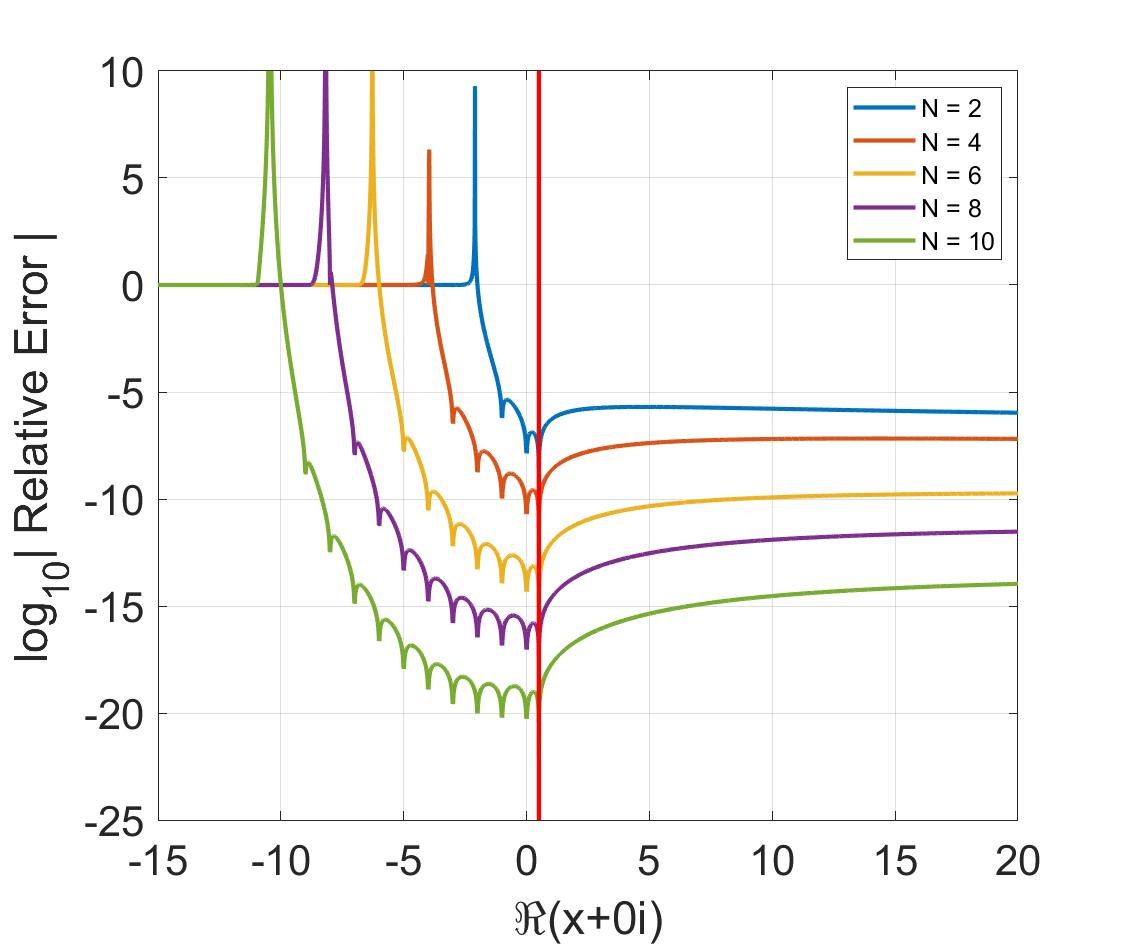}
	\includegraphics[width=0.48\linewidth]{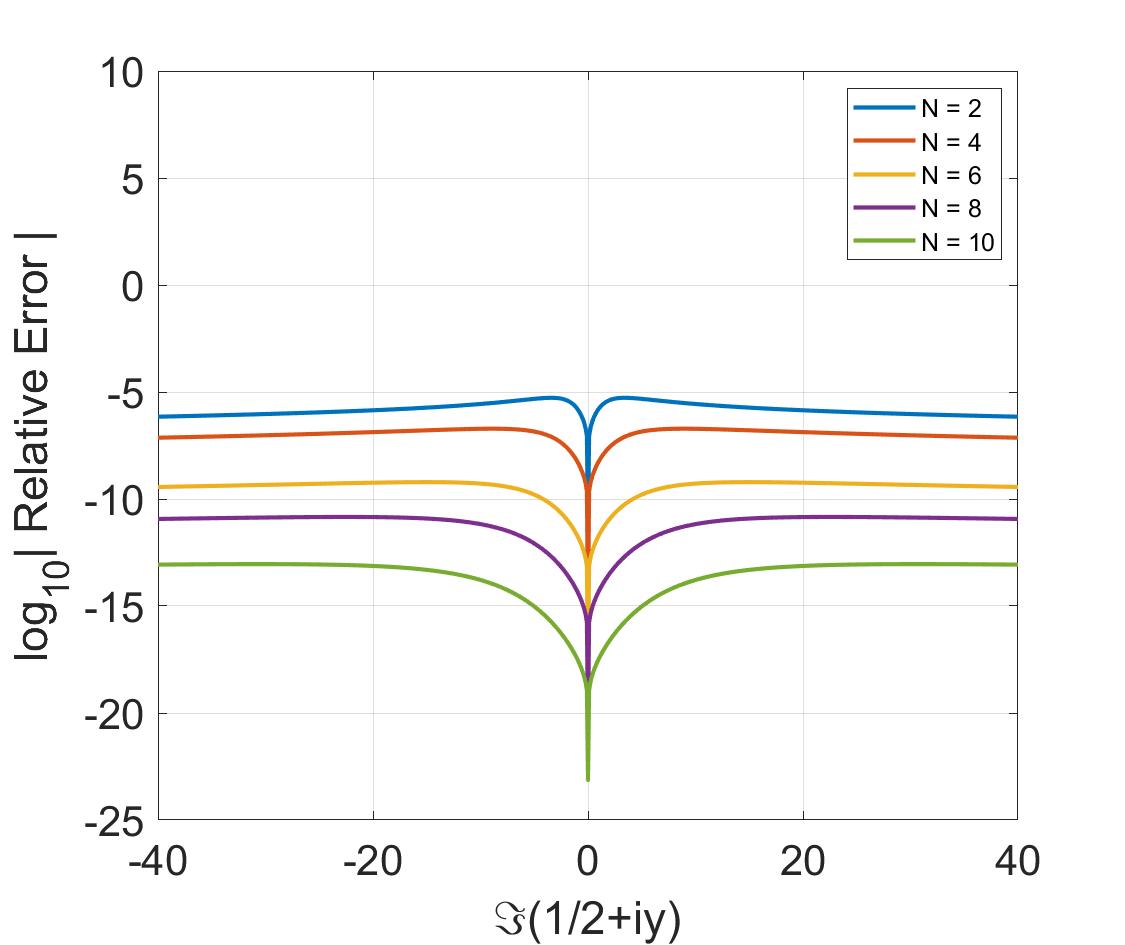}
	\includegraphics[width=0.48\linewidth]{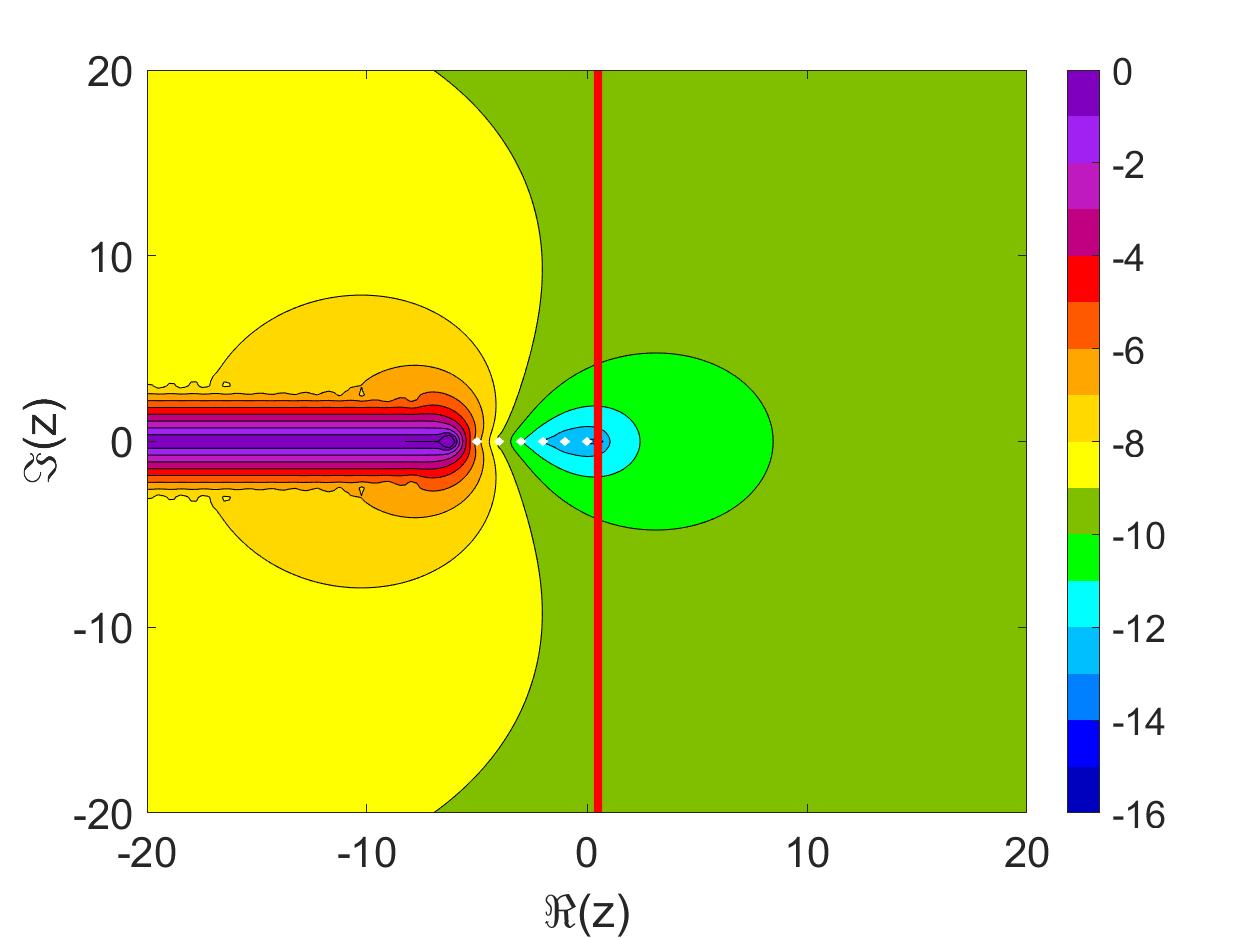}
	\includegraphics[width=0.48\linewidth]{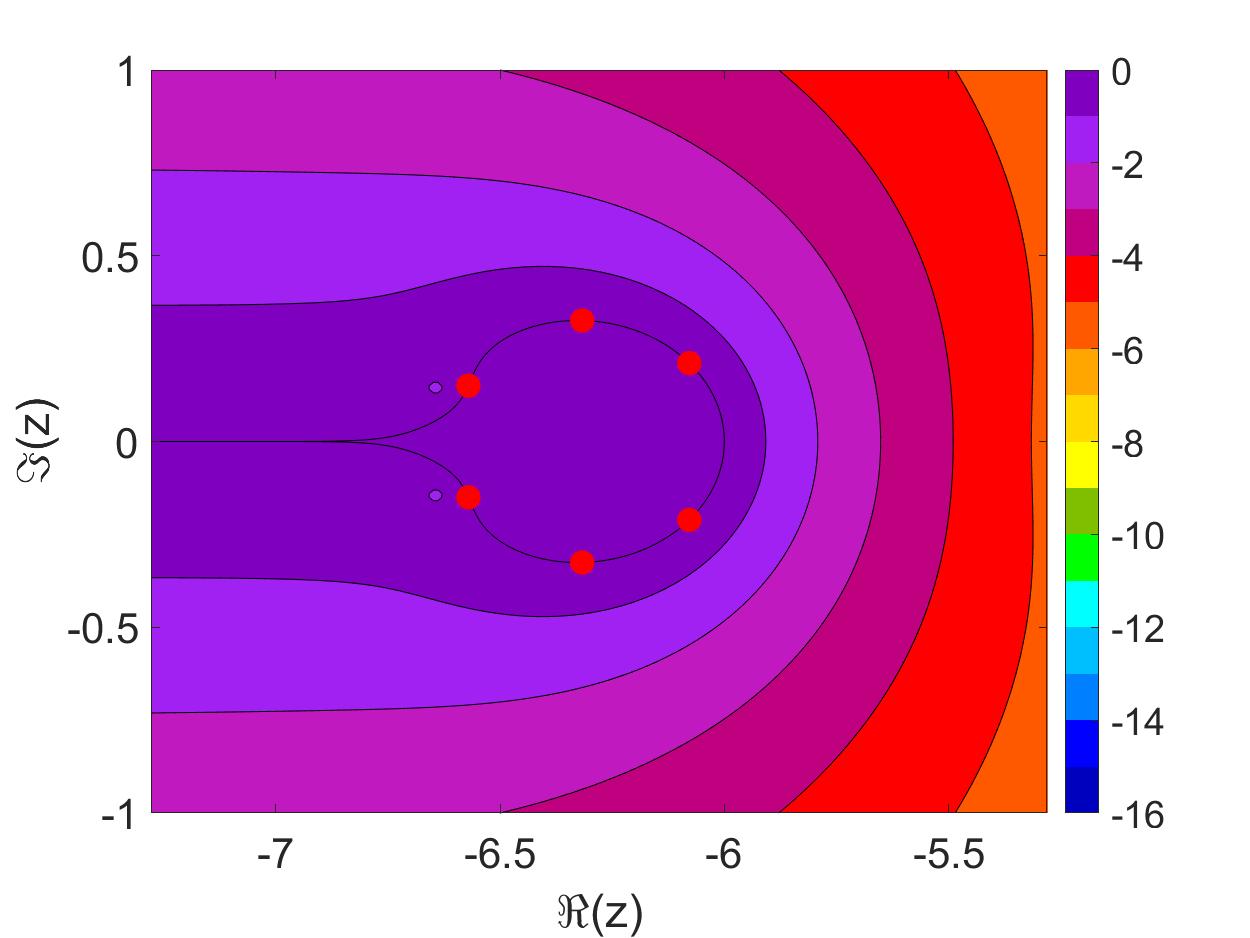}
	\caption{Relative error of $\Gamma(z)$ along the real line $z = x$ (top left), and the imaginary line $z=\frac{1}{2}+iy$ (top right) for Spouge's approximation \eqref{eqn:Spouge}. The approximation in the complex plane (bottom left) with $N=6$ poles exhibits a spurious branch cut, and spurious zeros near the branch point (zoom-in, bottom right).}
	\label{fig:Spouge}
\end{figure}
\begin{table}[ht!]
	\begin{center}
		\caption{The values of $r = r(\bar{z})$ which make Spouge's method exact at $z = \bar{z}$.}
		\label{tab:Spouge}
		\begin{tabular}{|c||c|c|c|c|c|}
			\hline
			$N$   & $r(0.5)$   & $r(1.0)$   & $r(2.0)$   & $r(50.0)$   & $r(100.0)$   \\ \hline 
			1 &   1.00185747 &   0.99223170 &   0.99194114 &   0.99138822 &   0.99119485 \\ \hline 
			2 &   2.09996567 &   2.10115467 &   2.10116440 &   2.10117352 &   2.10117352 \\ \hline 
			3 &   2.71663951 &   2.69959327 &   2.69870240 &   2.69689395 &   2.69622362 \\ \hline 
			4 &   3.96458259 &   3.95086119 &   3.94993006 &   3.94795127 &   3.94718590 \\ \hline 
			5 &   5.19066504 &   5.18533756 &   5.18484911 &   5.18373907 &   5.18328154 \\ \hline 
			6 &   6.27826689 &   6.28058958 &   6.28075657 &   6.28110947 &   6.28124423 \\ \hline 
			7 &   6.91355131 &   6.90188976 &   6.90073381 &   6.89803597 &   6.89689347 \\ \hline 
			8 &   8.17657066 &   8.16567447 &   8.16446857 &   8.16155303 &   8.16027376 \\ \hline 
			9 &   9.37725744 &   9.37476768 &   9.37442505 &   9.37353160 &   9.37310870 \\ \hline 
			10 &  10.44593152 &  10.44900164 &  10.44934278 &  10.45016390 &  10.45052136 \\ \hline 
		\end{tabular}
	\end{center}
\end{table}

In Figure \ref{fig:Spouge}, we choose $\bar{z} =\frac{1}{2}$, and construct the relative error $\log_{10}\left|1-\Gamma_S(z)/\Gamma(z)\right|$ of the Spouge approximation along the real line $z =x$ (top left), and the imaginary line $z = \frac{1}{2}+iy$ (top right). The errors are obtained using variable precision arithmetic. A contour plot over the complex plane is also depicted (bottom plots). For comparison with other methods below, we fix the degree $N = 6$, yieldinf a relative accuracy of 9 or more digits in the right half plane $\Re(z)>1/2$.

It is evident from the Figure that the approximation converges slowly as $|z|$ increases, in accordance with the bound (see theorem 1.3.1 in \cite{Spouge})
\[
	\left|1-\frac{\Gamma_S(z)}{\Gamma(z)}\right| \leq \frac{\sqrt{r+1}}{(2\pi)^{r+3/2}}\frac{1}{\Re(z+r)}.
\]
But we also see a striking feature, that the error vanishes at the first $N$ poles. In this sense, the Spouge approximation is interpolating $\Gamma(z)$ at the first $N$ non-positive integers (and at infinity).

Since the approximation of $F_r(z)$ is a rational approximation of degree $N$, it will produce $N$ zeros, which are spurious, as $\Gamma(z)\neq 0$ for $z\in \mathbb{C}$. As noted in \cite{AAA,Stahl,Aptekarev}, the zeros of a rational function written as a sum of poles satisfy
\[
	R(z) = R_\infty +\sum_{n=0}^{N-1} \frac{c_n}{z-p_n} = 0,
\]
which, along with the identity $\frac{z}{z-p_n} =1+\frac{p_n}{z-p_n}$ can be rewritten as
\[
	z = -\frac{1}{R_\infty}\sum_{n=0}^{N-1} \left(c_n+\frac{c_n p_n}{z-p_n}\right), \quad R_\infty \neq 0.
\]
Combining, we find that the zeros are the (nontrivial) eigenvalues of the $(N+1)\times(N+1)$ arrowhead matrix
\begin{equation}
	\label{eqn:Az}
	A = \begin{bmatrix}
			b_\infty& b_0	&\ldots	&b_{N-1}\\
			1 		& p_0	&\ldots	&0		\\
			\vdots	& 		&\ddots	&\vdots	\\
			1 		&  0	&\ldots	&p_{N-1}
	\end{bmatrix},
	\qquad b_\infty = -\frac{1}{R_\infty} \sum_{n=0}^{N-1} c_n, \qquad b_n = -\frac{c_n p_n}{R_\infty}.
\end{equation}
The zeros of the Spouge approximation are then computed as eigenvalues, and depicted in the bottom right panel of Figure \ref{fig:Spouge}. These zeros cluster around the branch point $z=-r$, a well-documented phenomenon of rational approximations for functions with branch point singularities \cite{Stahl,Aptekarev}.

\subsection{The Lanczos Approximation}
\begin{figure}[t!]
	\centering
	\includegraphics[width=0.48\linewidth]{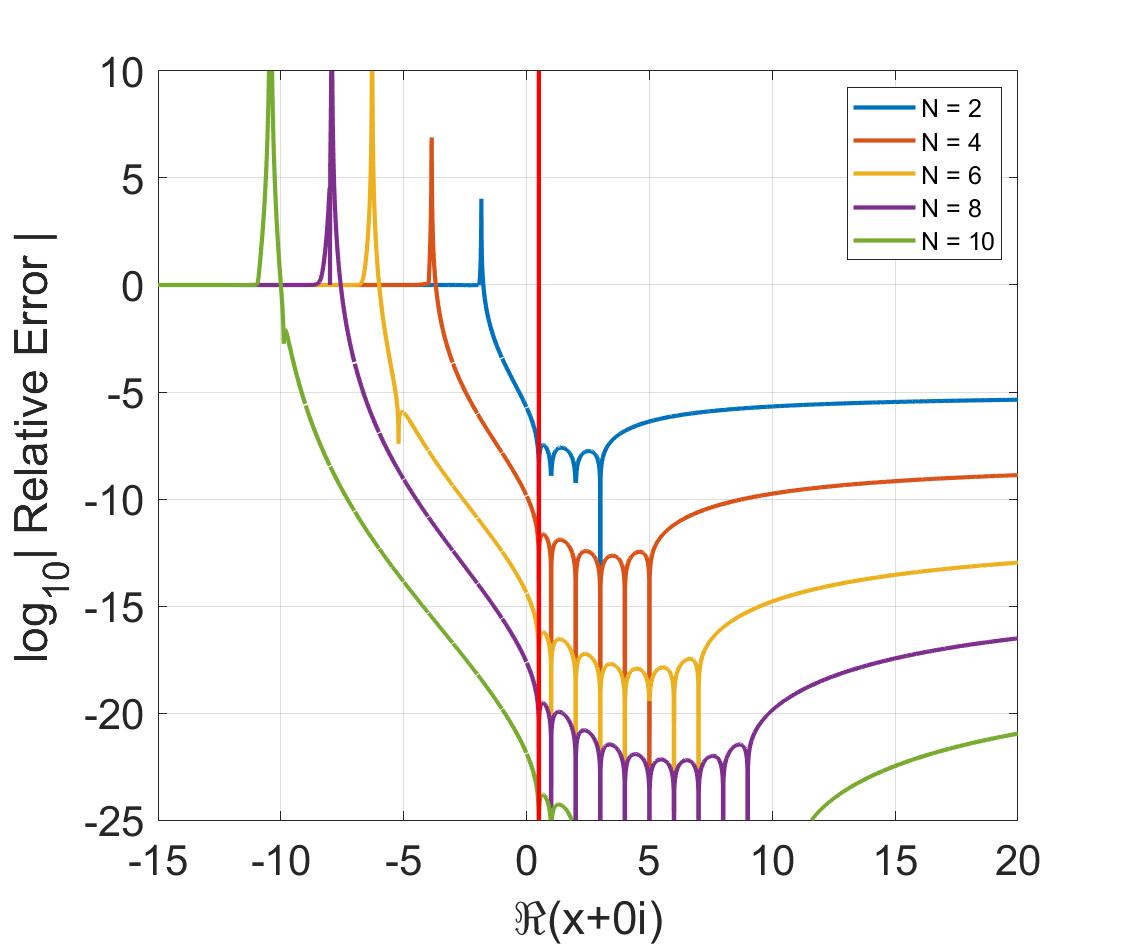}
	\includegraphics[width=0.48\linewidth]{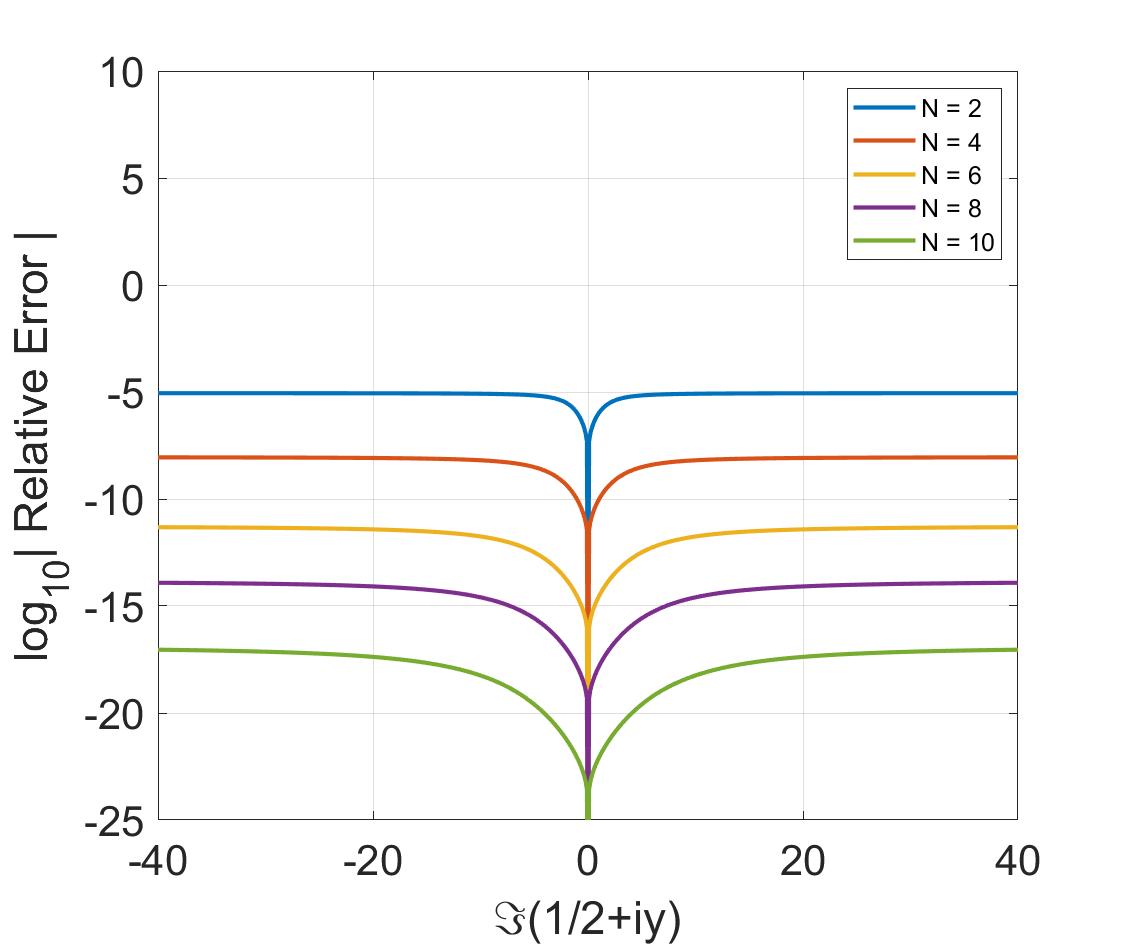}
	\includegraphics[width=0.48\linewidth]{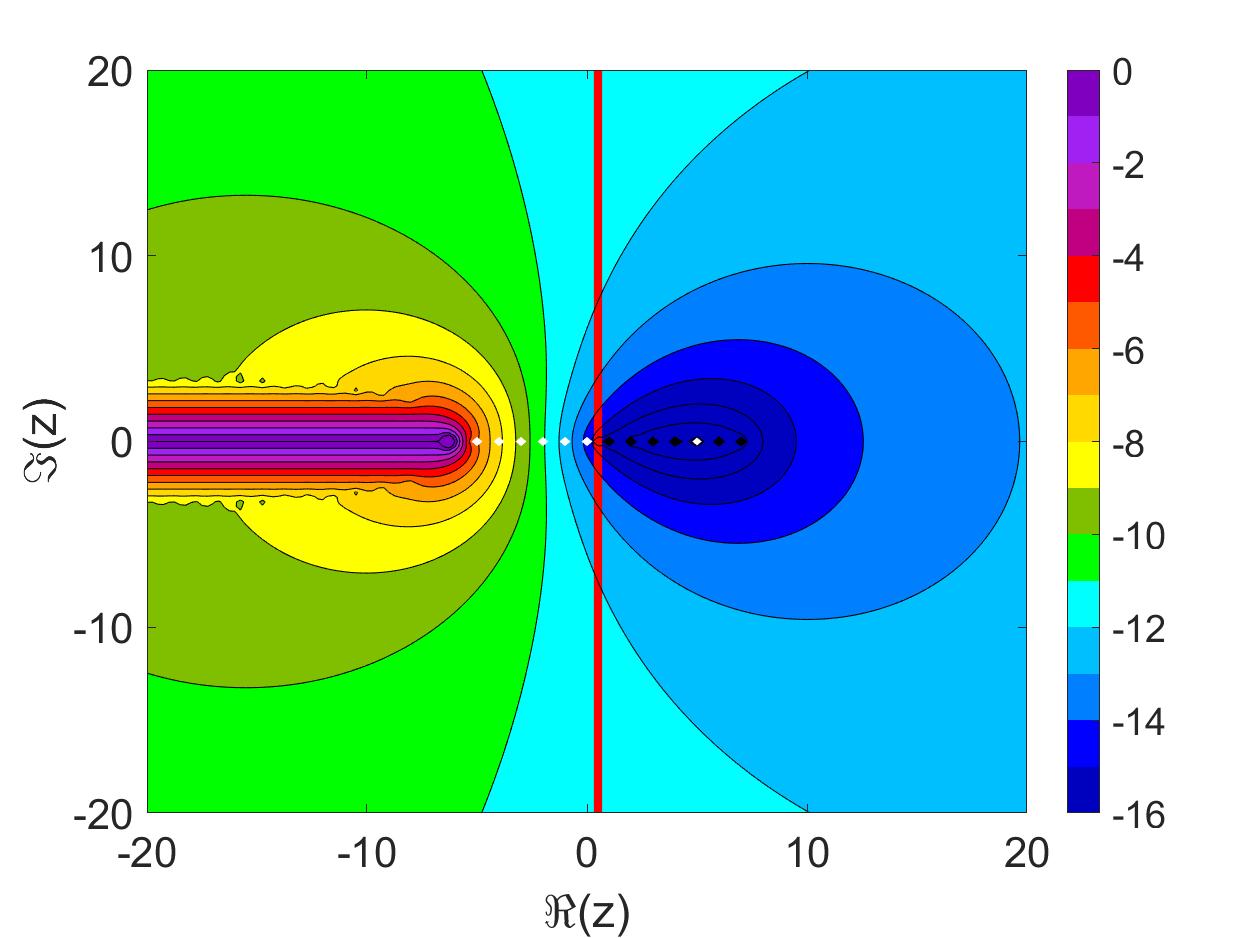}
	\includegraphics[width=0.48\linewidth]{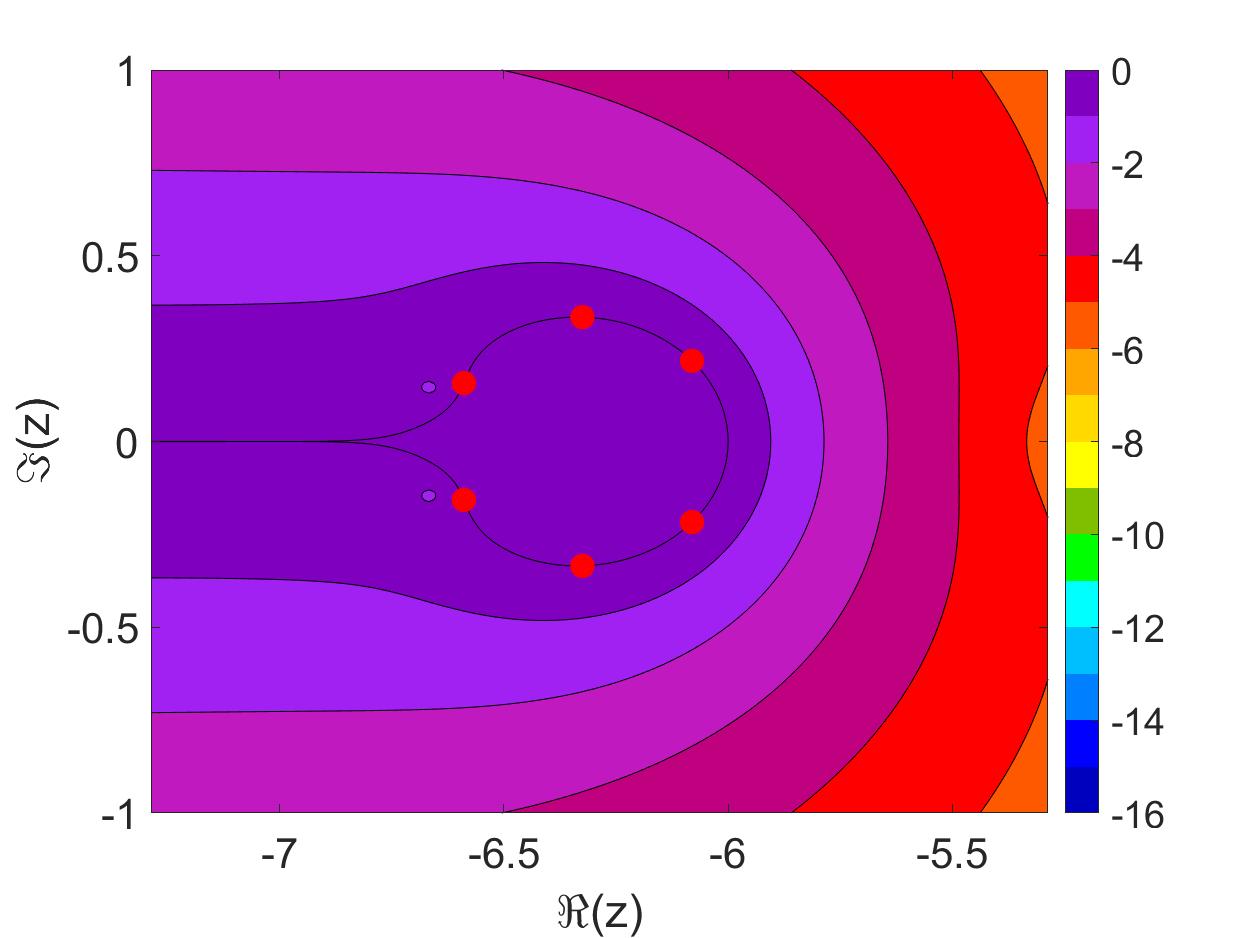}
	\caption{Relative error of $\Gamma(z)$ along the real line $z = x$ (top left), and the imaginary line $z=\frac{1}{2}+iy$ (top right) with the Lanczos approximation \eqref{eqn:Lanczos_H}. The approximation in the complex plane (bottom left) with $N=6$ poles exhibits a spurious branch cut, and spurious zeros near the branch point (zoom-in, bottom right).}
	\label{fig:Lanczos}
\end{figure}
The method in \cite{Lanczos} was studied extensively by Pugh \cite{Pugh}, and Luke \cite{Luke}, and has been made popular by its appearance in Numerical Recipes \cite{NumRec}. The Lanczos approximation first takes the form
\begin{equation}
	\label{eqn:Lanczos_H}
	\Gamma_L(z) = (z+r)^{z-\frac{1}{2}}e^{-(z+r)}\left[b_0(r) + \sum_{n=1}^{N} b_n(r)H_n(z)\right],
\end{equation}
using the rational functions (c.f. the recursion, eqn. \eqref{eqn:rec})
\begin{equation}
	\label{eqn:H}
	H_n(z) = \frac{\Gamma^2(z)}{\Gamma(z+n)\Gamma(z-n)} = \frac{(z-1)(\ldots)(z-n)}{z(\ldots)(z+n-1)}.
\end{equation}
The expansion coefficients were obtained by a series of very creative arguments by Lanczos involving a Chebyshev series expansion of \eqref{eqn:Gamma_Def}, yielding (in our notation)
\begin{equation}
	\label{eqn:b_Lanczos}
	b_0(r) = \frac{e^{r+1}}{\sqrt{r+1}} = F_r(1), \qquad b_n(r) = 2n\sum_{k=0}^n (-1)^{n-k}\frac{(n+k-1)!}{(k!)^2(n-k)!} F_r(k+1).
\end{equation}
The functions $H_n(z)$ are then re-expanded by Lanczos into a sum of poles, with a form identical to \eqref{eqn:Spouge} but where the coefficients $c_n(r)$ differ from the true residues.

\begin{table}[h!]
	\begin{center}
		\caption{The values of $r = r(\bar{z})$ which make Lanczos' method exact for various $\bar{z}$.}
		\label{tab:Lanczos}
		\begin{tabular}{|c||c|c|c|c|c|c|}
			\hline
			$N$   & $r(0.5)$   & $r(15.0)$   & $r(20.0)$   & $r(50.0)$   & $r(100.0)$   & $r(\infty)$  \\ \hline 
			1 &   1.00077330 &   0.99051561 &   0.99020468 &   0.98961285 &   0.98940582&  0.989194 \\ \hline 
			2 &   2.10377552 &   2.10350676 &   2.10344648 &   2.10331555 &   2.10326432&  2.103209 \\ \hline 
			3 &   3.13999099 &   3.15268144 &   3.15324285 &   3.15435549 &   3.15475877&  3.155180 \\ \hline 
			4 &   3.86444531 &   3.84659270 &   3.84540501 &   3.84289319 &   3.84192626&  3.840882 \\ \hline 
			5 &   5.10339071 &   5.08750951 &   5.08622486 &   5.08339096 &   5.08225558&  5.081000 \\ \hline 
			6 &   6.28671094 &   6.28217746 &   6.28169594 &   6.28055659 &   6.28006828&  6.279506 \\ \hline 
			7 &   7.37615402 &   7.37831319 &   7.37846649 &   7.37878053 &   7.37889490&  7.379012 \\ \hline 
			8 &   7.92985725 &   7.91507505 &   7.91348318 &   7.90967267 &   7.90801798&  7.906094 \\ \hline 
			9 &   9.17946385 &   9.16597042 &   9.16437818 &   9.16044255 &   9.15867667&  9.156578 \\ \hline 
			10 &  10.41889651 &  10.40882965 &  10.40750139 &  10.40407749 &  10.40247363& 10.400511 \\ \hline 
		\end{tabular}
	\end{center}
\end{table}
In Figure \ref{fig:Lanczos} we take $r(\bar{z}) = r(1/2)$, which reduces the error along the line of symmetry. In particular, we are able to obtain better than 11 decimals of relative precision over the right half plane $\Re(z)>1/2$ with $N=6$ poles (bottom left panel). As was the case with Spouge's approximation, the spurious zeros cluster near the branch point $z=-r$ (bottom right panel).

Again, we see that the error vanishes at several points, but this time at the positive integers. It is likely that Lanczos was aware of the interpolating property of his approximation. However, we emphasize that it is not mentioned explicitly in \cite{Lanczos}. We therefore state this as a theorem.
\begin{theorem}
	The Lanczos approximation \eqref{eqn:Lanczos_H}, with \eqref{eqn:H} and \eqref{eqn:b_Lanczos}, is exact at $z= 1, 2, \ldots, N+1$.
\end{theorem}
\begin{proof}
	The approximation \eqref{eqn:Lanczos_H} satisfies
	\begin{align*}
	\Gamma(z) =(z+r)^{z-1/2}e^{-(z+r)}\left[\sum_{n=0}^N b_n(r)H_n(z) + \epsilon(z;r,N)\right],
	\end{align*}
	where the error is of the form
	\[
	\epsilon(z;r,N) = \sum_{j=1}^\infty b_{N+j}(r)H_{N+j}(z).
	\]
	But the rational functions \eqref{eqn:H} satisfy (c.f. eqn. \eqref{eqn:rec})
	\begin{align*}
	H_{N+j}(z) &= \frac{\Gamma(z)}{\Gamma(z-N-1)}\frac{\Gamma(z)}{\Gamma(z+N+j)}\frac{\Gamma(z-N-1)}{\Gamma(z-N-j)} \\
	&= (z-1)(\ldots)(z-N)\left[\frac{(z-N-1)(\ldots)(z-N-j)}{(z)(\ldots)(z+N-1+j)}\right], \qquad j\geq 1.
	\end{align*}
	It follows that for each $j\geq 1$, $H_{N+j}(z)=0$ for $z=1, 2, \ldots N+1$, and thus the error vanishes identically.
\end{proof}
\begin{remark}
	Luke \cite{Luke2} and Pugh \cite{Pugh} both use the property that $H_k(m) = 0$ for $m = 1, 2, \ldots, k$ to recursively compute the coefficients $b_n(r)$ in \eqref{eqn:Lanczos_H}. They do not explore the interpolation properties any further.
\end{remark}

\begin{remark}
	The Lanczos commits a relative error that decreases geometrically with $N$,
	\[
	\left|1-\frac{\Gamma_L(z)}{\Gamma(z)}\right|\leq  C\rho^{-N}, \qquad \rho>1.
	\]
	See Luke \cite{Luke2} for details. Pugh \cite{Pugh} further observed that the error is predominantly due to the first term in the infinite sum,
	\[
		\epsilon(z;r,N) \approx b_{N+1}(r)H_{N+1}(z).
	\]
\end{remark}

In Table \eqref{tab:Lanczos} we record $r(\bar{z})$ corresponding to the Lanczos method. The final column contains $r(\infty) = r^*-1/2$, where the values of $r^*$ are taken from Appendix C of \cite{Pugh}, and the adjustment is due to notational differences.


\section{Reformulation}
We will now generalize the Lanczos and Spouge methods as follows. The asymptotically scaled gamma function \eqref{eqn:Fr} can be well-approximated by a sum of poles placed at the non-positive integers. If we sample $F_r(z)$ at $N+1$ points $z=z_j$ in half-plane $\Re(z+r)>0$, then the coefficient vector $\textbf{c}$ satisfies the nearly-Cauchy system $C\textbf{c}=\textbf{f}$, as defined above in \eqref{eqn:Hf}. The procedure is summarized as follows.
\begin{enumerate}
	\item Choose the parameter $r>N-1$. Typically, $r =r(\bar{z})$ can be selected to interpolate an additional point $\bar{z}$, which may or may not be $\infty$.
	\item Choose a set of $N+1$ values $z_j$.  Set up and solve the linear system \eqref{eqn:Hf}.
	\item Form the approximation, valid for $\Re(z+r)>0$
	\begin{equation}
		\label{eqn:F_N}
		\Gamma(z) \approx \Gamma_N(z;r) = (z+r)^{z-1/2}e^{-(z+r)}F_N(z;r), \qquad \text{with} \qquad F_N(z;r) = c_\infty(r) + \sum_{n=0}^{N-1}\frac{c_n(r)}{z+n}.
	\end{equation}
\end{enumerate}
\begin{remark}
	Typically we choose the points to be positive integers. But we can relax this constraint, provided we are willing to precompute several values $\Gamma(z_j)$ in the complex plane, to a suitably high precision. We can also over-sample, and construct a least-squares fit for the coefficients.
\end{remark}

\begin{remark}
	In the numerical demonstrations below, the relative error is constructed using high precision, to illustrate convergence beyond $N=6$. The nearly-Cauchy matrix $C$ is ill-conditioned, and we observe a loss of one or two decimals of accuracy using double precision when $N=6$.
\end{remark}

We now examine the error in the approximation.
\begin{theorem}
	Let $\Gamma_N(z;r)$ be given as in \eqref{eqn:F_N}. Then in the half-disk $\Re(z)>0$, $|z| \leq \max\{|z_j|\}$, the relative error satisfies
	\[
		\left|1-\frac{\Gamma_N(z;r)}{\Gamma(z)}\right| \leq C_N(r)\left|\left(\frac{1}{z}-\frac{1}{\bar{z}}\right)\frac{\psi(z)}{\phi(z)}\right|,
	\]
	where
	\[
		\phi(z) = z(z+1)(\ldots)(z+N-1), \qquad \psi(z) = \prod_{j=1}^{N+1}(z-z_j).
	\]
	The constant $C_N(r)$ does not depend on $z$. The decay at $\infty$ is characterized by the asymptotic behavior
	\[
		1-\frac{\Gamma_N(z;r)}{\Gamma(z)} \sim  1-\frac{c_\infty(r)}{\sqrt{2\pi}}+\frac{D_N(r)}{z}, \qquad |z|\gg 1.
	\]
\end{theorem}

\begin{proof}
	We first write the approximation \eqref{eqn:F_N} as
	\[
		\Gamma_N(z;r) = (z+r)^{z-\frac{1}{2}}e^{-(z+r)}\frac{P_N(z;r)}{\phi(z)}, \qquad \phi(z) =\frac{\Gamma(z+N)}{\Gamma(z)}=z(z+1)(\ldots)(z+N-1).
	\]
	Thus the polynomial $P_N(z;r) = \phi(z)F_N(z;r)$ is of degree $N$, and satisfies
	\[
		P_N(z) = G_N(z;r)(1+\epsilon_N(z;r)), \qquad G_N(z;r) = \frac{\Gamma(z+N)e^{z+r}}{(z+r^{z-\frac{1}{2}})}=\phi(z)F(z;r),
	\]
	where the relative error is
	\[
		\epsilon_N(z;r) = 1-\frac{\Gamma_N(z;r)}{\Gamma(z)} = 1-\frac{P_N(z;r)}{G_N(z;r)}.
	\]
	Applying \eqref{eqn:Hf}, we find $P_N(z_j,r) = G_N(z_j;r)$ for $j=1, 2, \ldots, N+1$. By choosing $r =r(\bar{z})$, the approximation is also made exact at $z=\bar{z}$ which is possibly $\infty$, and so the interpolation error is of the form
	\[
		G_N(z;r) - P_N(z;r) = \tilde{C}_N(r)\left(\frac{1}{z}-\frac{1}{\bar{z}}\right)\psi(z), \qquad \psi(z) = \prod_{j=1}^{N+1}(z-z_j),
	\]
	for some constant $\tilde{C}_N(r)$ that is independent of $z$. Thus, the relative error has the form
	\[
		\epsilon_N(z;r) = \frac{\tilde{C}_N(r)}{G_N(z;r)}\left(\frac{1}{z}-\frac{1}{\bar{z}}\right)\psi(z)
						= \frac{\tilde{C}_N(r)}{F(z;r)}\left(\frac{1}{z}-\frac{1}{\bar{z}}\right)\frac{\psi(z)}{\phi(z)}.
	\]
	But for $\Re(z)>0$, $F(z;r)$ is bounded and non-zero, and so by choosing a new constant $C_N(r)$ appropriately, we obtain the first result.
	
	Now, for large $|z|$, we apply Stirling's formula to $\Gamma(z)$, and expand the poles in $F_N(z;r)$ to find
	\[
		F(z;r) \sim \sqrt{2\pi}\frac{z^{z-\frac{1}{2}}e^{-r}}{(z+r)^{z-\frac{1}{2}}}\left[1+\frac{1}{12z}+\mathcal{O}\left(\frac{1}{z^2}\right)\right], \qquad F_N(z;r) \sim c_\infty(r)\left[1 + \frac{\alpha(r)}{z}+\mathcal{O}\left(\frac{1}{z^2}\right)\right],
	\]
	where
	\[
		\alpha(r) = \frac{1}{c_\infty(r)}\sum_{n=0}^{N-1}c_n(r).
	\]
	Noting that $(1+r/z)^{z-\frac{1}{2}}\sim e^r$, we obtain
	\begin{align*}
		1-\frac{\Gamma_N(z;r)}{\Gamma(z)} &= 1 - \frac{F_N(z;r)}{F(z;r)} \\
		&\sim 1 - \frac{c_\infty(r)}{\sqrt{2\pi}}\left[\frac{1+\frac{\alpha(r)}{z}}{1+\frac{1}{12z}}\right]+\mathcal{O}\left(\frac{1}{z^2}\right) \\
		&\sim 1-\frac{c_\infty(r)}{\sqrt{2\pi}}+\left(\alpha(r)-\frac{1}{12}\right)\frac{1}{z} + \mathcal{O}\left(\frac{1}{z^2}\right).
	\end{align*}
	Upon setting $D(r) = \alpha(r)-\frac{1}{12}$, we obtain the second result.
\end{proof}

\subsection{Numerical Implementation}
\label{Imp}
Once a value of $r$ is chosen and the coefficients are determined, the approximation $\Gamma_N(z)$ is valid in the complex plane, away from the branch cut emanating from $z=-r$ along the negative real axis. In practice, evaluation is only necessary for $\Re(z)\geq \frac{1}{2}$, due to the Euler reflection formula \eqref{eqn:Ref}. We also consider several numerical concerns.
\begin{enumerate}
	\item The computation for $|z|\gg 1$ often leads to numerical overflow along the positive real line, and underflow elsewhere in the complex plane. To mitigate this issue, we compute the approximation via
	\begin{equation}
	\label{eqn:Gamma_Implement}
	\Gamma_N(z;r) = e^{\phi(z)}F_N(z;r), \qquad \phi(z) =\left(z-\frac{1}{2}\right)\ln(z+r)-z-r,
	\end{equation}
	where $F_N(z;r)$ is from \eqref{eqn:F_N}. In particular, we note that $\phi(z)$ is of a computable magnitude when the gamma function is near overflow or underflow limitations. 
	\item For $\Re\left(z\right)<\frac{1}{2}$, we use Euler's reflection formula \eqref{eqn:Ref}, and instead compute $\Gamma(1-z)$ with \eqref{eqn:Gamma_Implement}, i.e.
	\[
		\Gamma(z) \approx \frac{\pi}{\sin(\pi z)\Gamma_N(1-z;r)}, \qquad \text{for $\Re(z)<\frac{1}{2}$}.
	\]
	We also take care to avoid loss of precision due to range reduction in the sine function. If $m<\Re(z)\leq m+1$, write $z = -m + \delta z$, and compute
	\[
		\sin(\pi z) = \sin(\pi (-m+\delta z)) = (-1)^m \sin(\pi \delta z).
	\]
\end{enumerate}

\section{Numerical Tests}
We will now test several interpolation strategies, and examine their behavior. The errors will be constructed using variable precision arithmetic, and thus the errors can be examined for large and small $|z|$. We also omit the use of the reflection formula to study the errors near the branch cut. In practice, the computations would produce symmetric errors about the line of symmetry, with a slight loss of accuracy near the poles (see \cite{Rump} for a detailed discussion). Experiments in double precision arithmetic produced a degradation of accuracy due for $N>6$ due to the ill-conditioned system \eqref{eqn:Hf}, and loss of precision.

\subsection{Fixed poles with interpolation along the postive real axis}
\begin{figure}[hb!]
	\centering
	\includegraphics[width=0.32\linewidth]{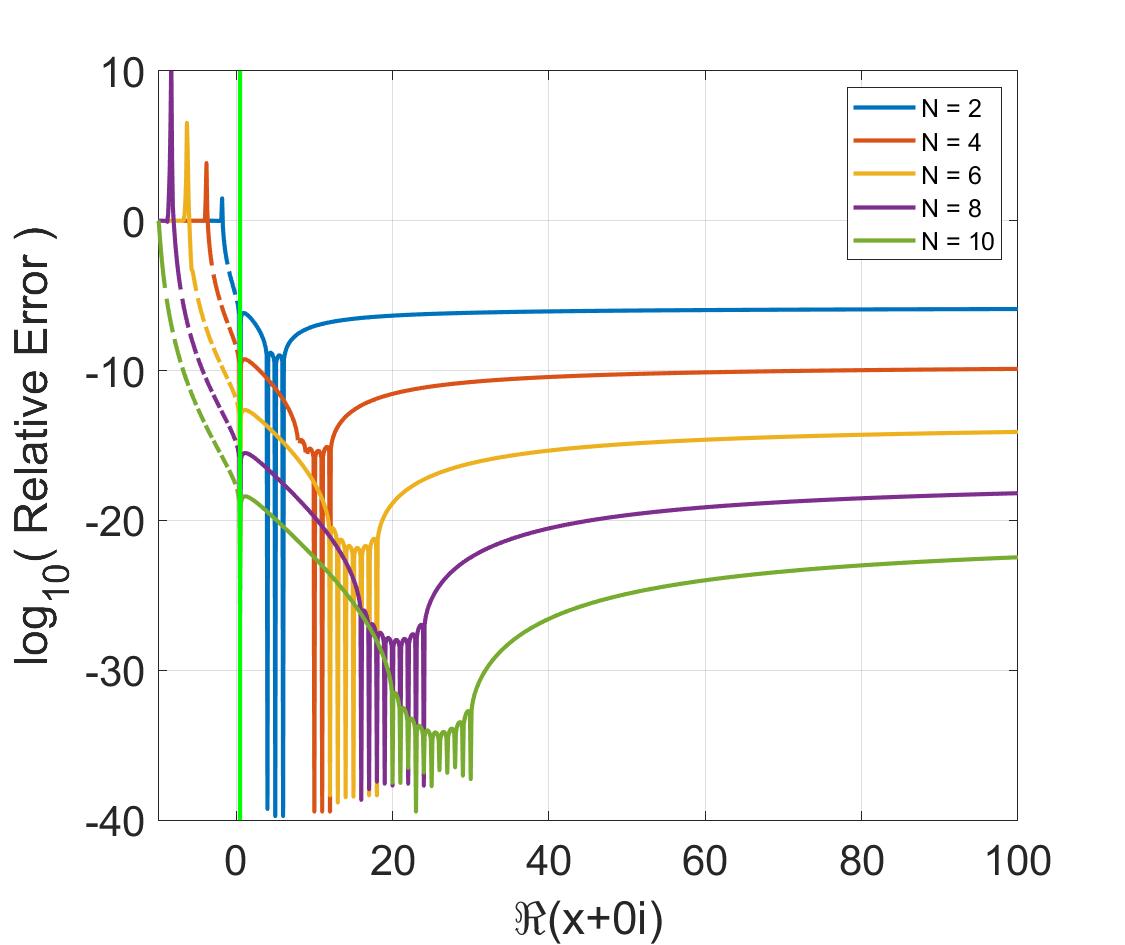}
	\includegraphics[width=0.32\linewidth]{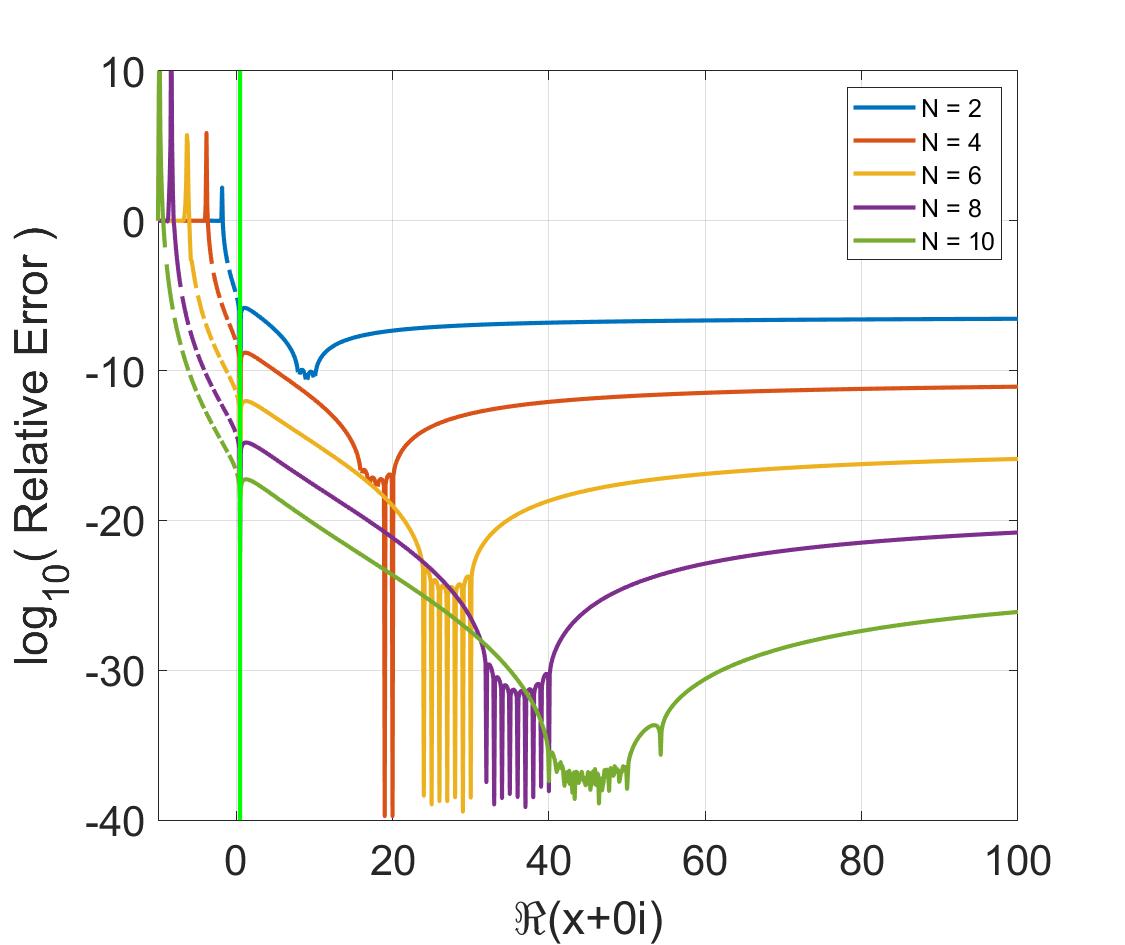}
	\includegraphics[width=0.32\linewidth]{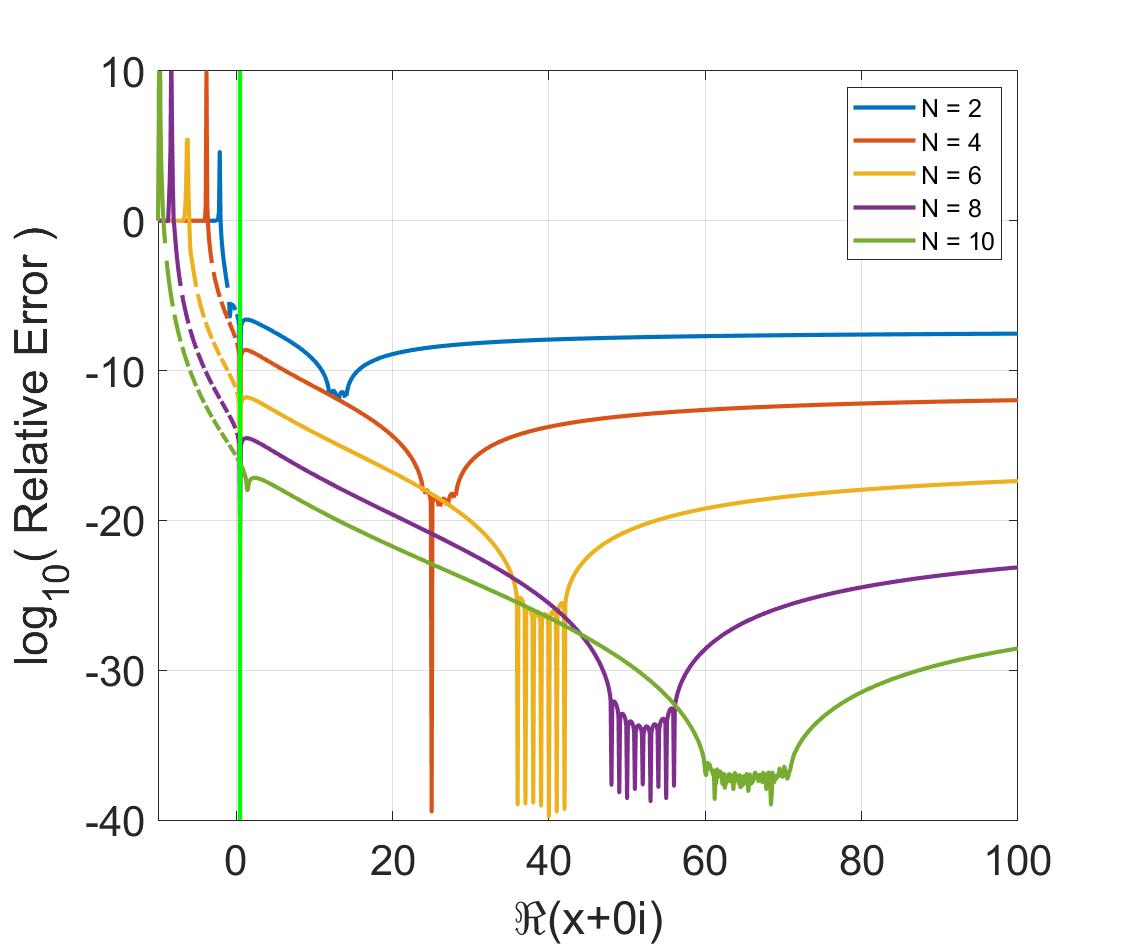}
	\includegraphics[width=0.32\linewidth]{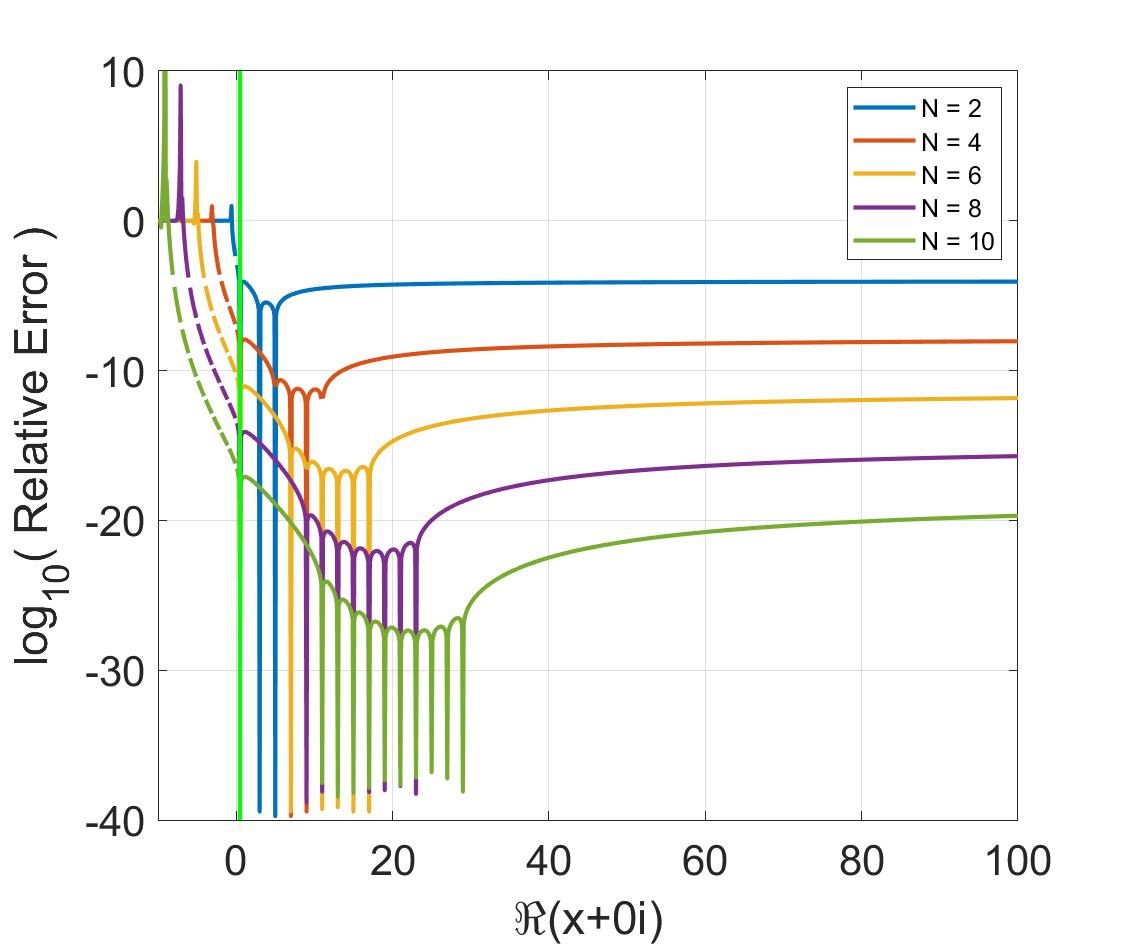}
	\includegraphics[width=0.32\linewidth]{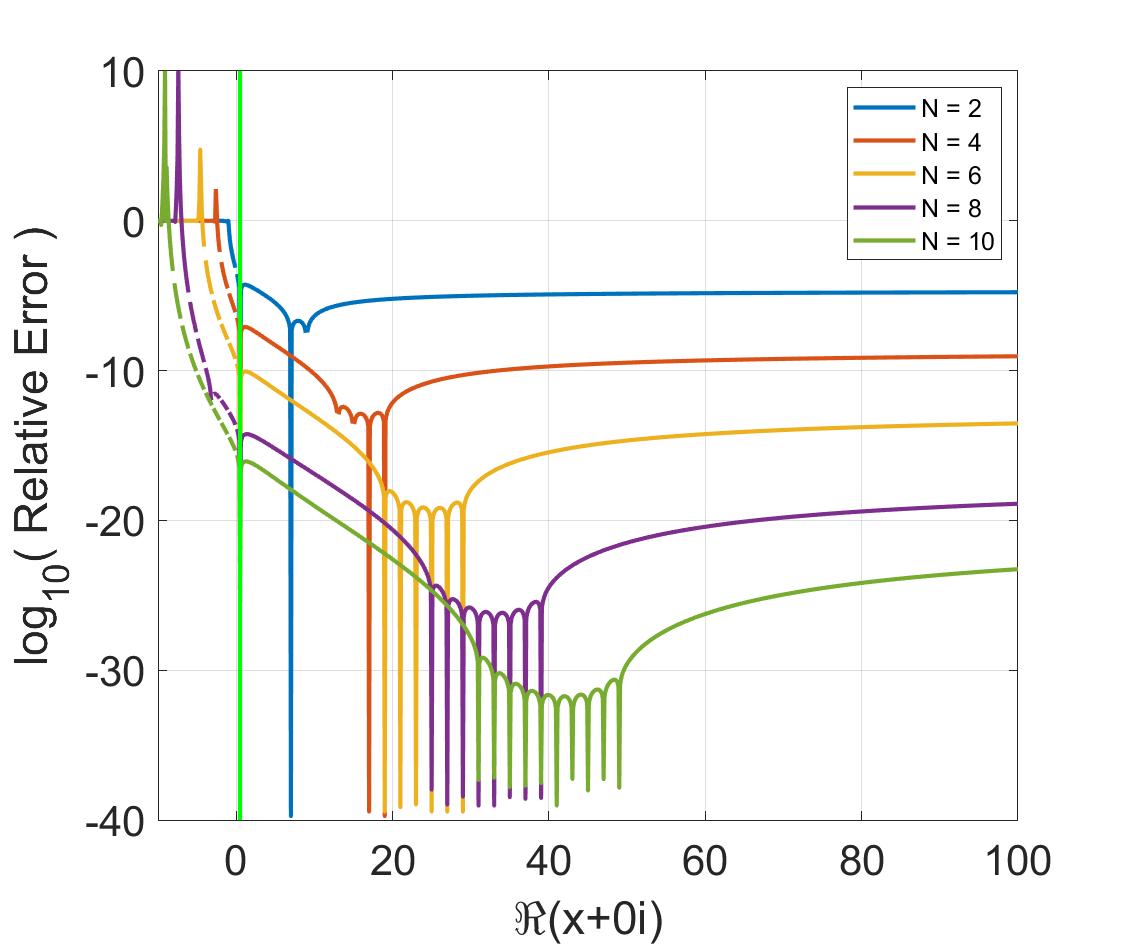}
	\includegraphics[width=0.32\linewidth]{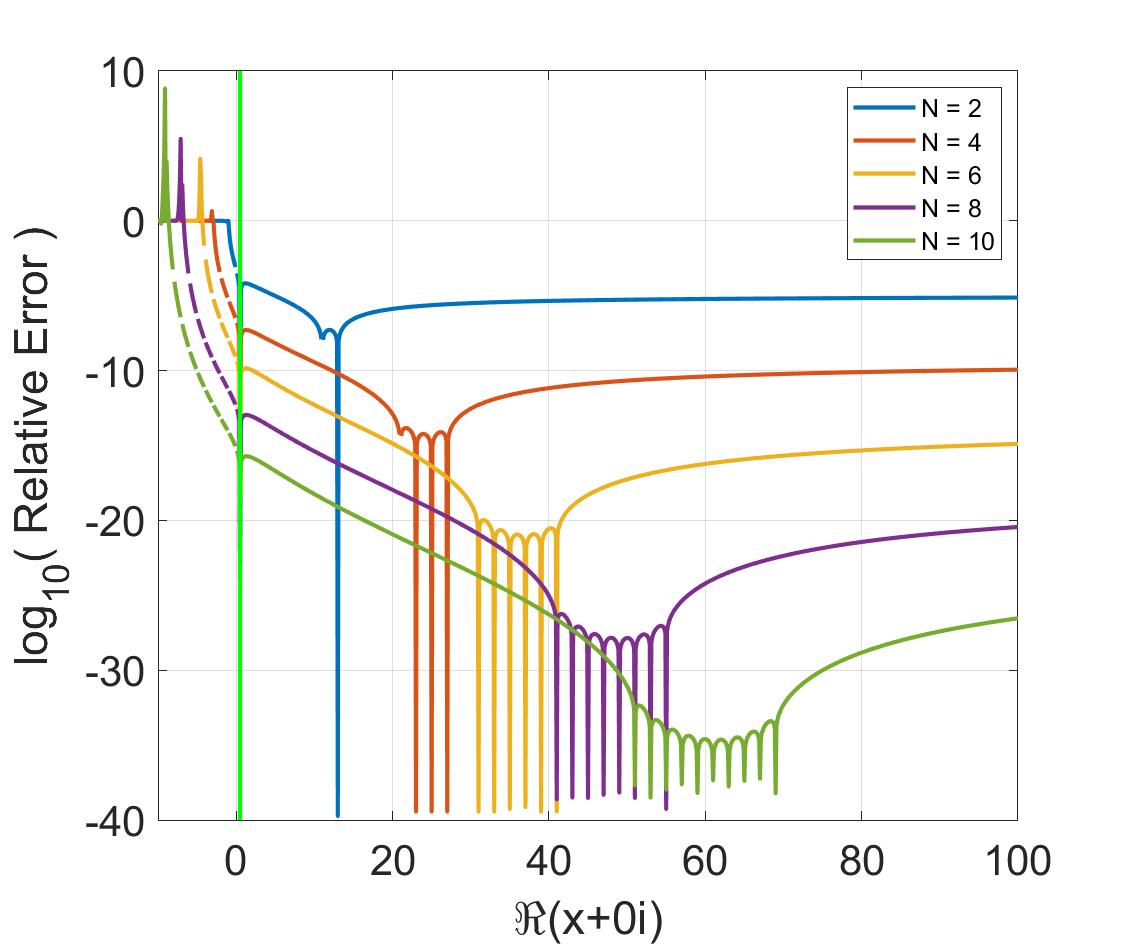}
	\includegraphics[width=0.32\linewidth]{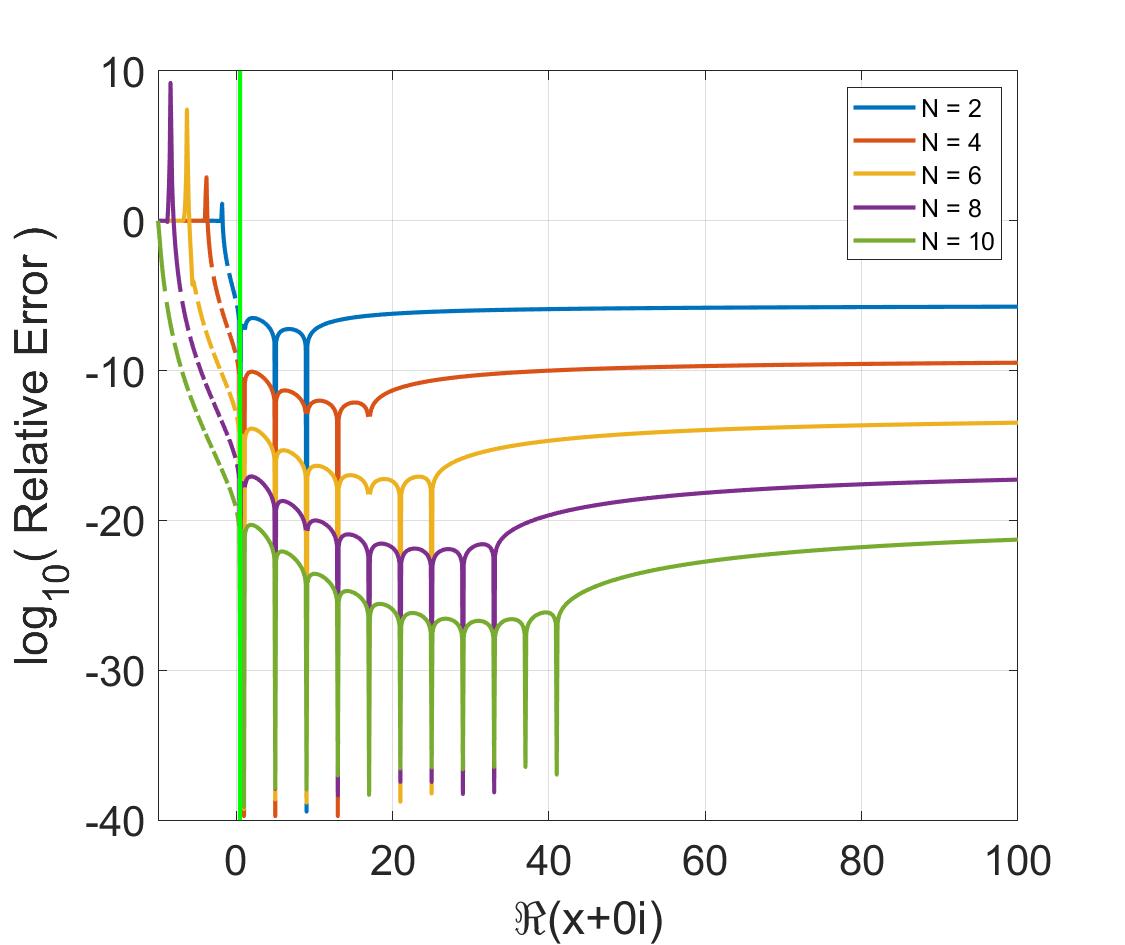}
	\includegraphics[width=0.32\linewidth]{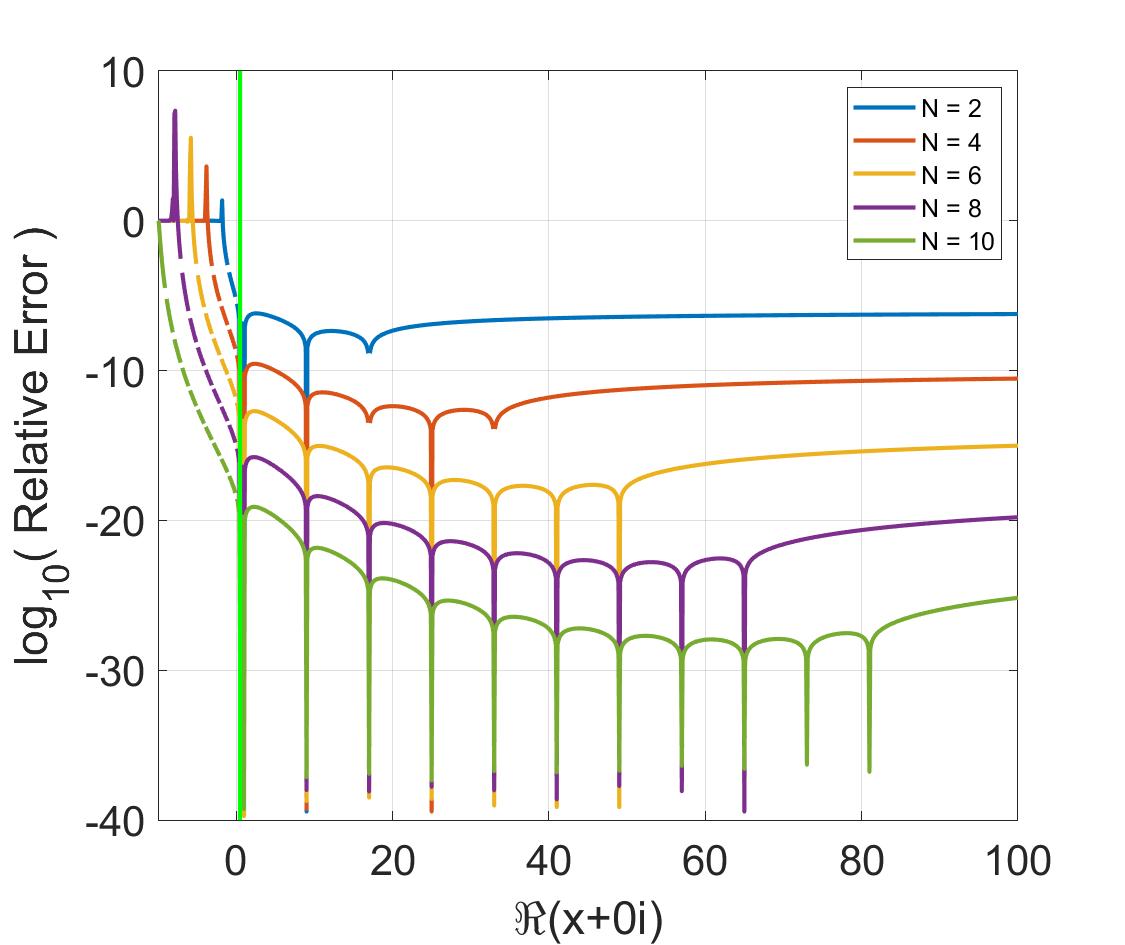}
	\includegraphics[width=0.32\linewidth]{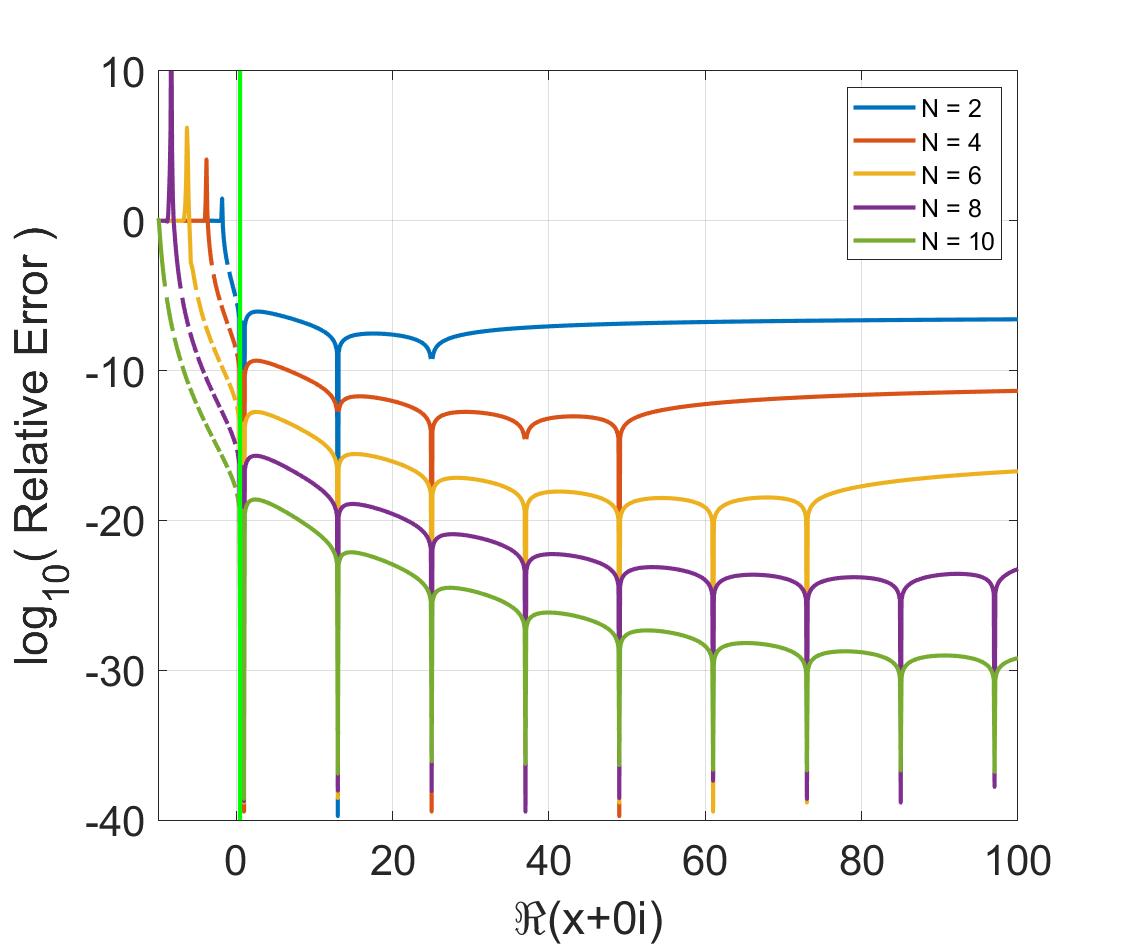}
	\caption{Relative error of $\Gamma(z)$ along the real line $z = x$, with interpolation at the the consecutive integers (top), narrowly spaced integers (middle) and widely spaced integers (bottom).}
	\label{fig:I1}
\end{figure}
The Lanczos approximation takes the interpolation points to be the first $N+1$ consecutive integers. We first consider several generalizations to this choice, where the points are integers dispersed over the real line. They are:
\begin{itemize}
	\item no spacing:  $\{z_n\}_{n=0}^{N} = n+k+1$, for $k=2, 4, 6$ (top row, Figure \ref{fig:I1}).
	\item narrow spacing:  $\{z_n\}_{n=0}^{N} = 2n+k+1$, for $k=2, 4, 6$  (middle row, Figure \ref{fig:I1}).
	\item wide spacing:   $\{z_n\}_{n=0}^{N} = kn+1$, for $k=2, 4, 6$ (bottom row, Figure \ref{fig:I1}).
\end{itemize}
\begin{figure}[hb!]
	\centering
	\includegraphics[width=0.48\linewidth]{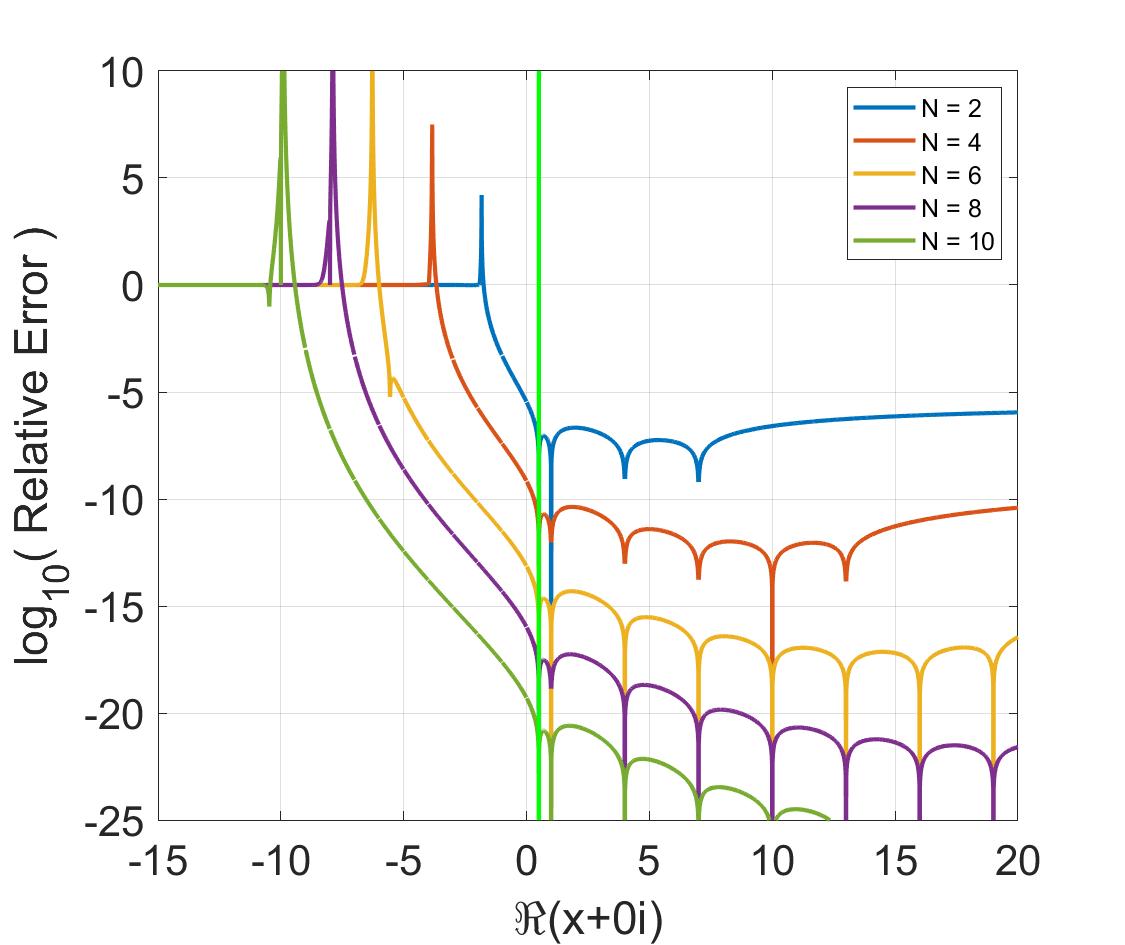}
	\includegraphics[width=0.48\linewidth]{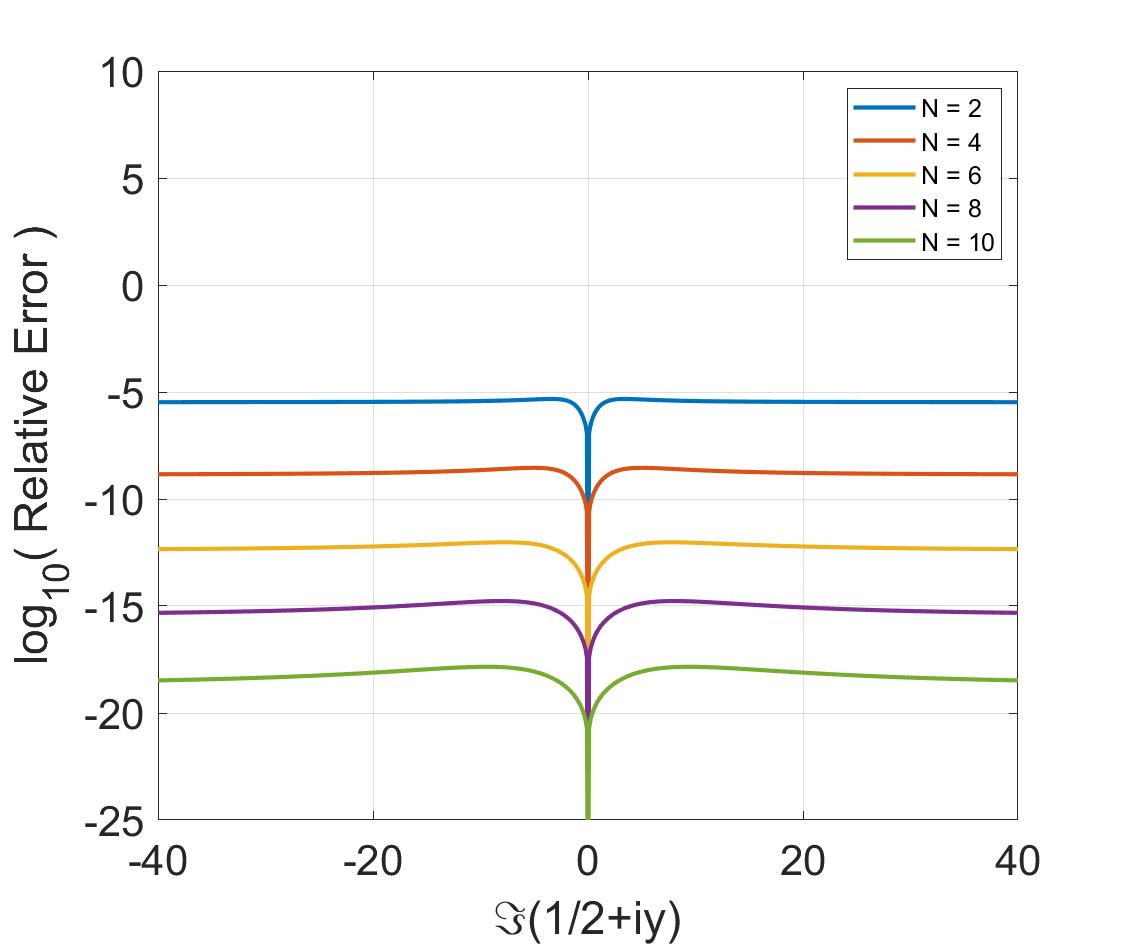}
	\includegraphics[width=0.48\linewidth]{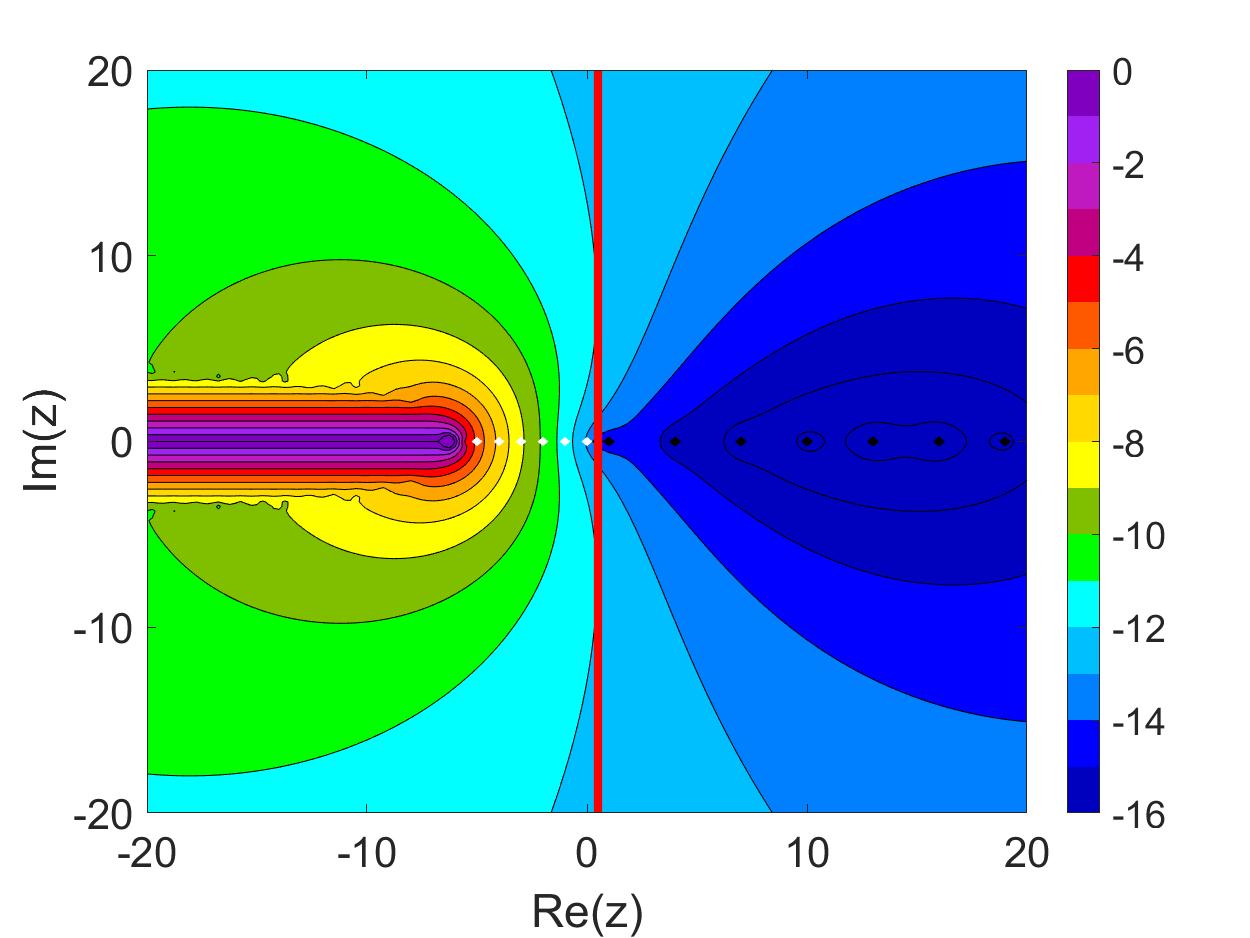}
	\includegraphics[width=0.48\linewidth]{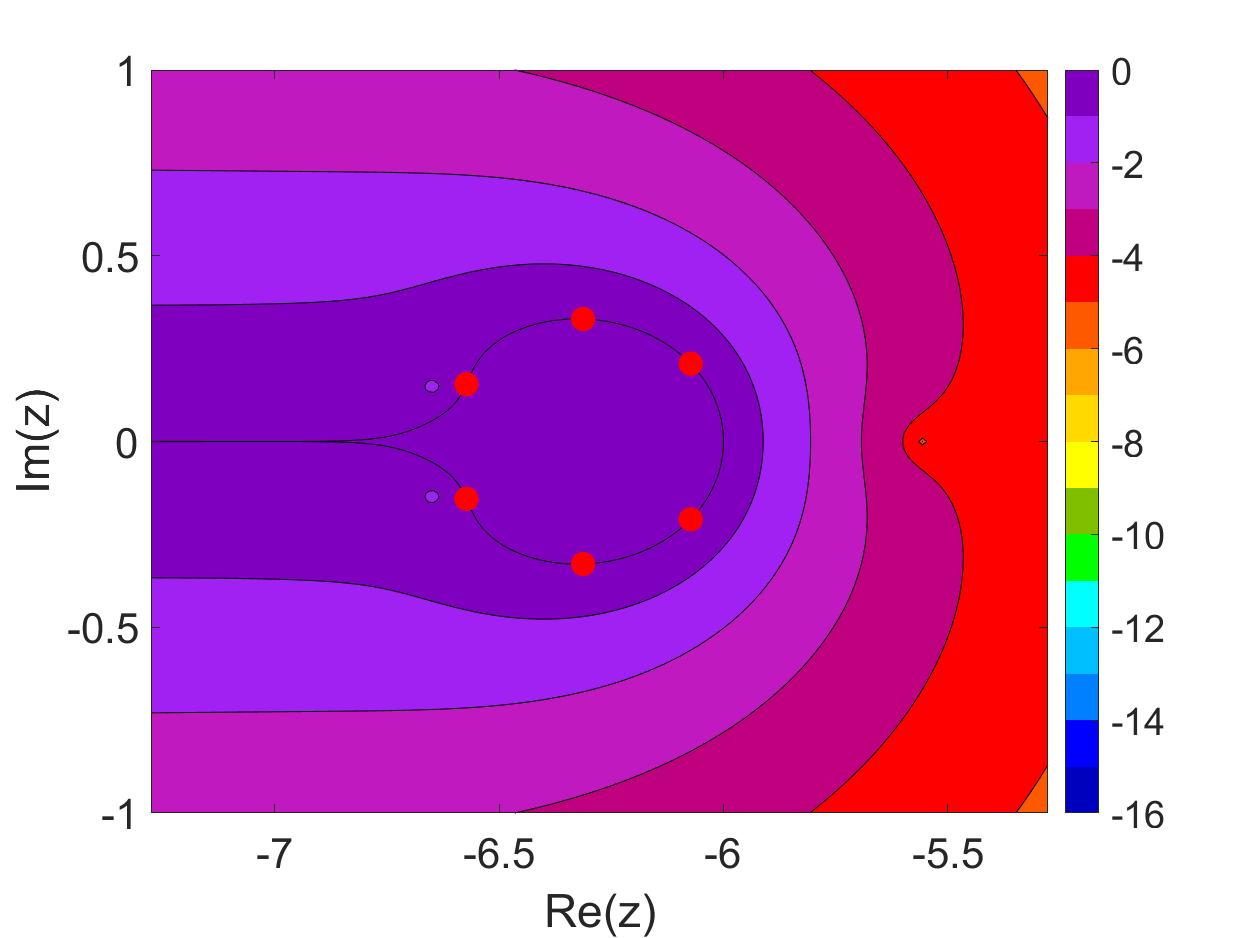}
	\caption{Relative error of $\Gamma(z)$ along the real line $z = x$ (top left), and the imaginary line $z=\frac{1}{2}+iy$ (top right) with interpolation at the points $1, 4, \ldots, 3N+1$. The approximation in the complex plane (bottom left) with $N=6$ poles exhibits a spurious branch cut, and spurious zeros near the branch point (zoom-in, bottom right).}
	\label{fig:N2}
\end{figure}
As the spacing becomes wider, so too does the proper interpolating region along the real axis. This is however at the expense of convergence elsewhere in the complex plane, particularly along the line of symmetry.

The overall best fit using interpolation along the real line for $N=6$ poles was found to use $z = \{1,4,7,10,13,16,19\}$. In Figure \ref{fig:N2} we construct this approximation, choosing $r=6.276394363877011$ to produce $\Gamma(\frac{1}{2})$ to high precision. The approximation achieves slightly more than 12 decimals along the line of symmetry, and better than 12 further into the right half plane. (bottom panels of Figure \ref{fig:N2}). The accuracy is slighlty better than the Lanczos approximation.

\subsection{Fixed poles with interpolation along the line of symmetry}
In Figure \ref{fig:Imag_M}, we select interpolation points along the line of symmetry $z = \frac{1}{2}+iy$, choosing the points $\{y_n\}_{n=0}^N = -3N+6n$ as odd integers in conjugate pairs. Since these points will already contain $z = \frac{1}{2}$, we determine $r$ so that $\Gamma(z)$ is exact at  $z = \bar{z} = 1$. As shown in the lower panels of Figure \ref{fig:Imag_M}, the approximation with $N=6$ poles and $r = 6.270484017574683$ produce the most uniform error yet, achieving more than 12 decimals of precision over most of the right half plane. The spurious zeros again cluster near the branch point, and are largely indistinguishable from the distribution in Figures \ref{fig:Spouge}, \ref{fig:Lanczos} and \ref{fig:N2}.

\begin{figure}[ht!]
	\centering
	\includegraphics[width=0.48\linewidth]{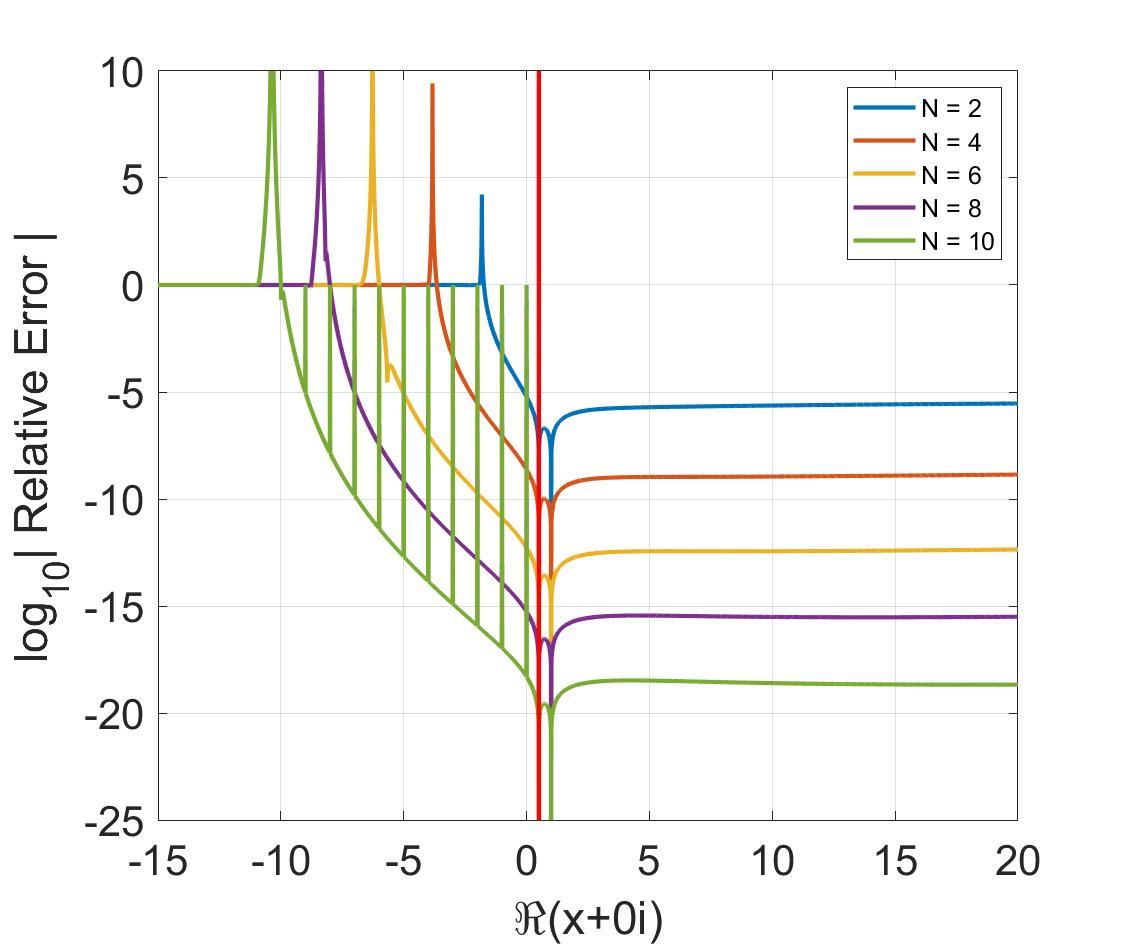}
	\includegraphics[width=0.48\linewidth]{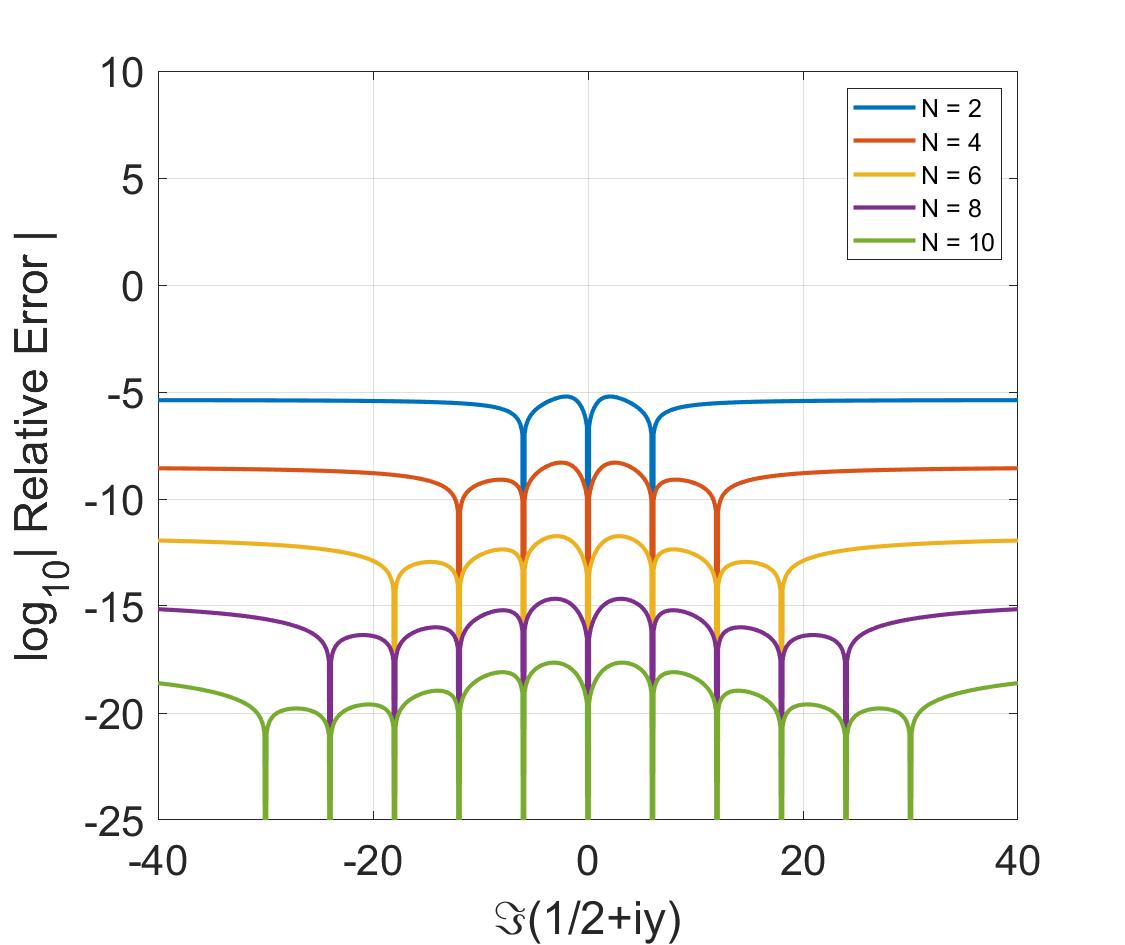}
	\includegraphics[width=0.48\linewidth]{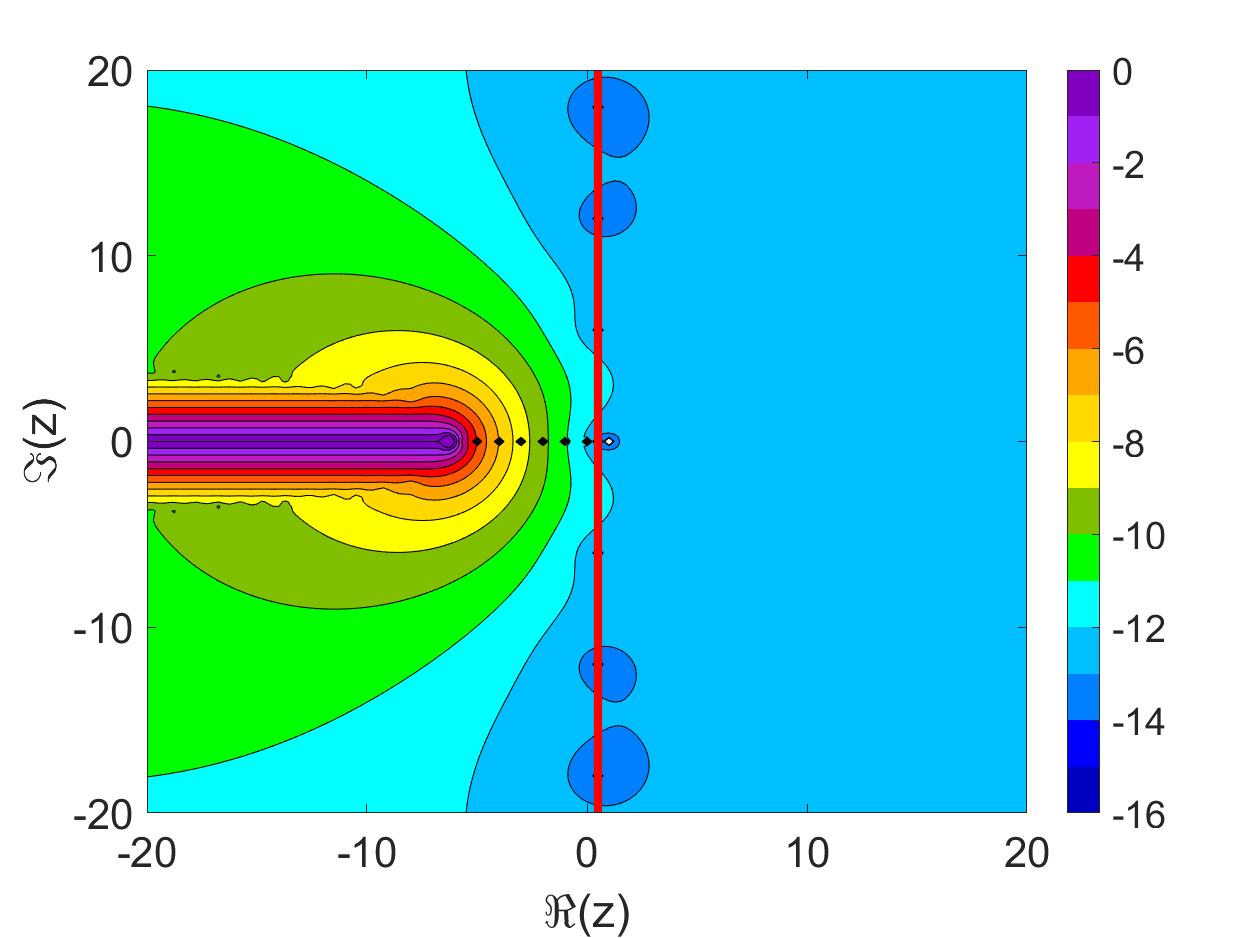}
	\includegraphics[width=0.48\linewidth]{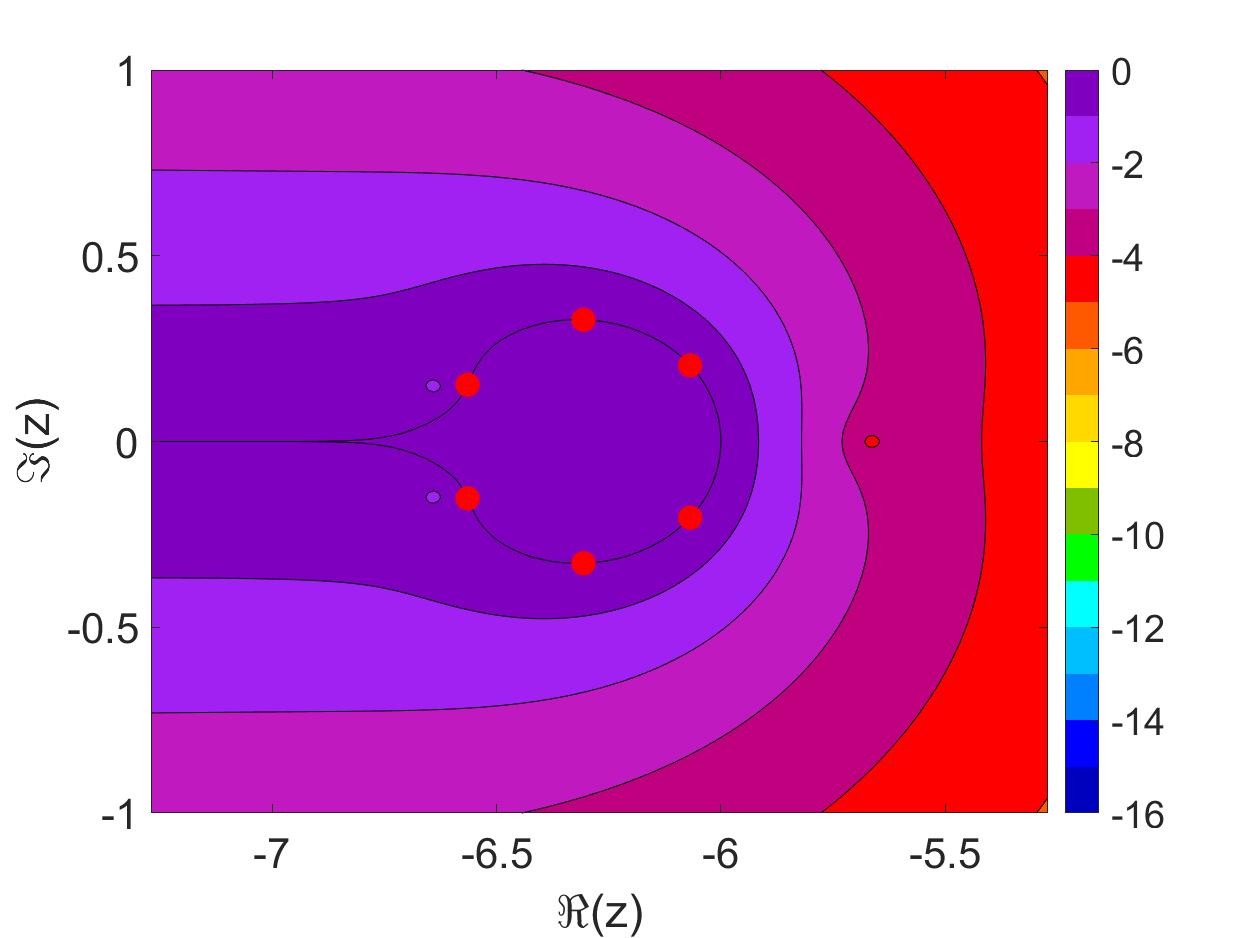}
	\caption{Relative error of $\Gamma(z)$ along the real line $z = x$ (top left), and the imaginary line $z=\frac{1}{2}+iy$ (top right) with interpolation along the imaginary line $z = \frac{1}{2}+iy$. The approximation in the complex plane (bottom left) with $N=6$ poles exhibits a spurious branch cut, and spurious zeros near the branch point (zoom-in, bottom right).}
	\label{fig:Imag_M}
\end{figure}

\subsection{Rational approximation with unspecified poles}
\begin{figure}[ht!]
	\includegraphics[width=0.48\textwidth]{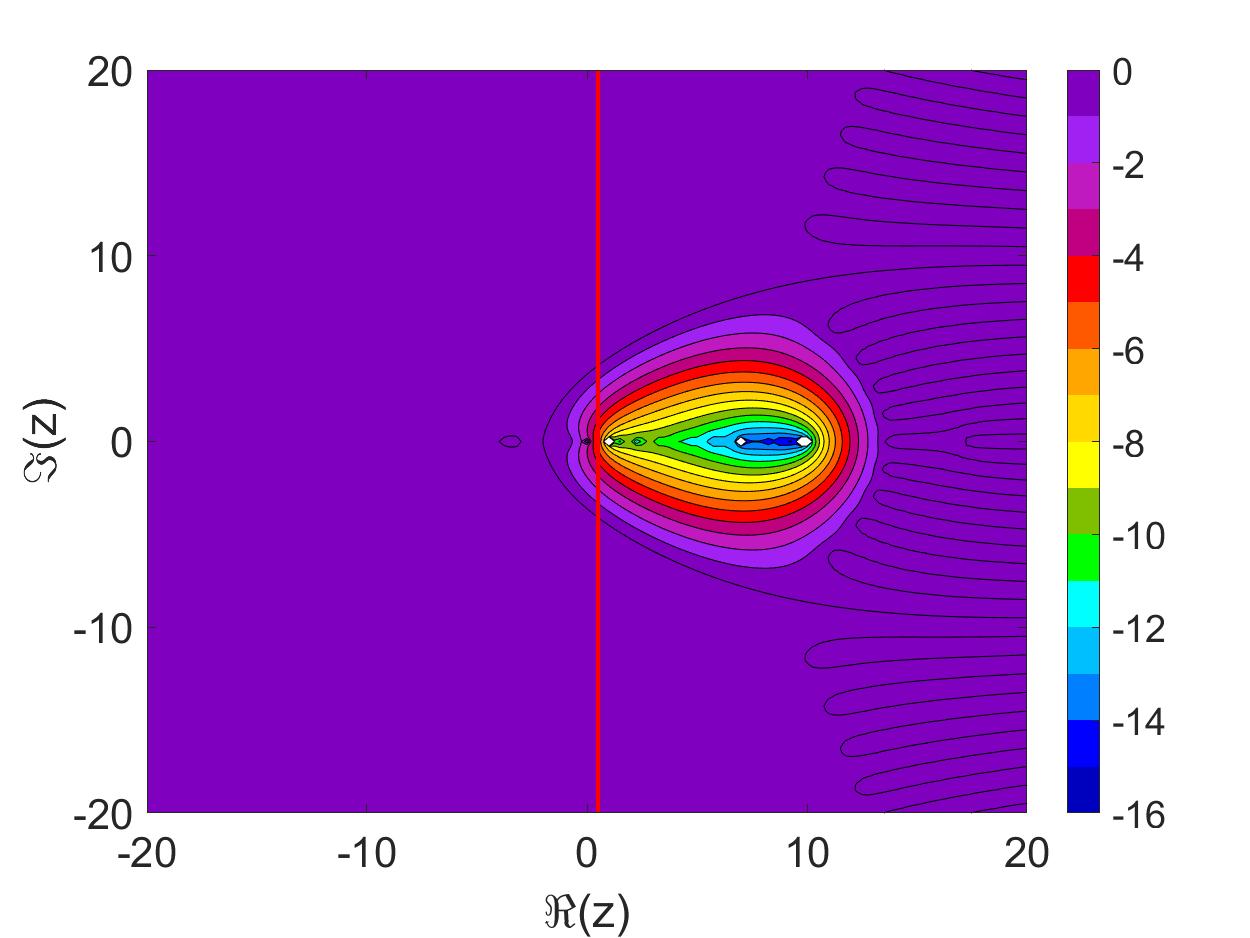}
	\includegraphics[width=0.48\textwidth]{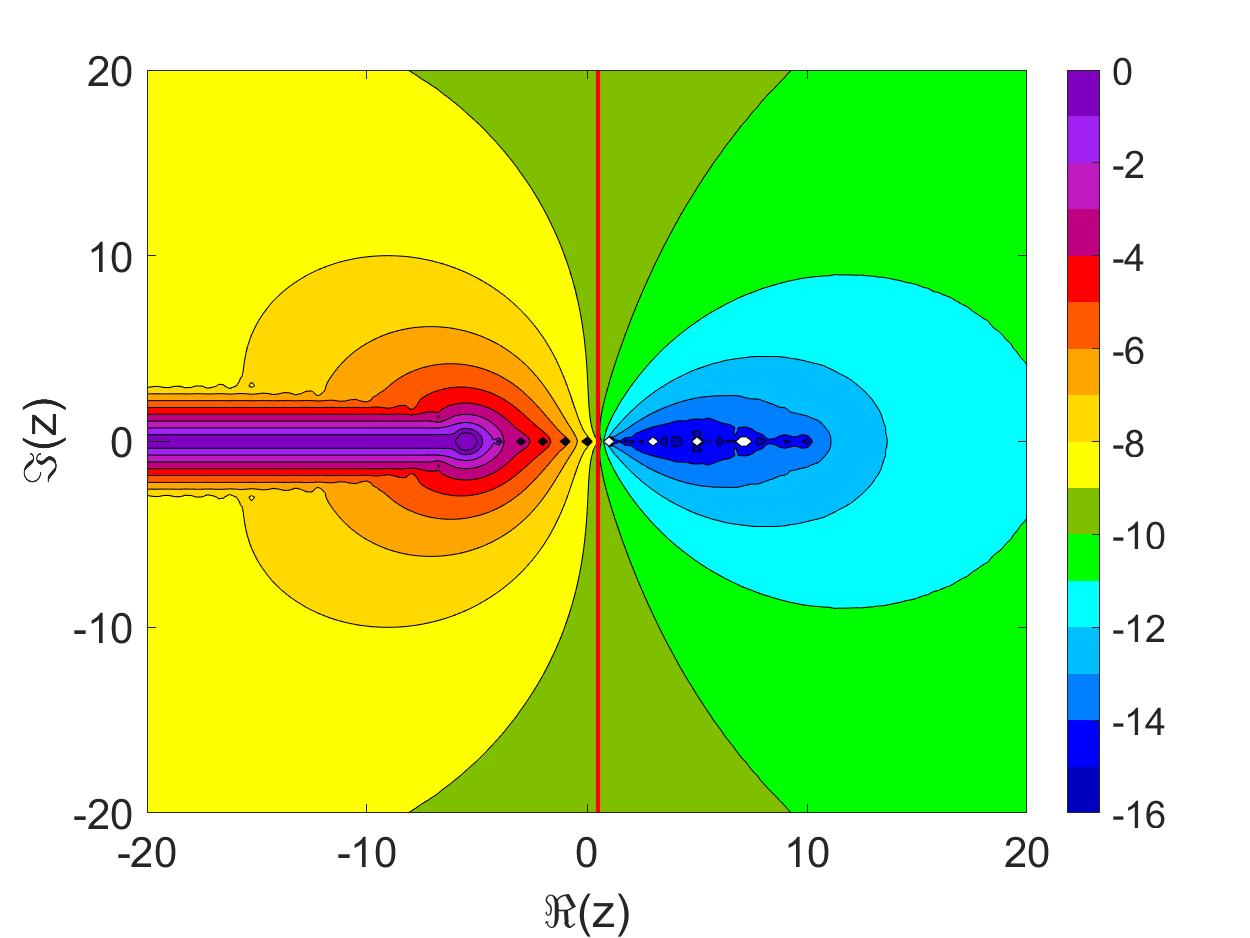}
	\includegraphics[width=0.48\textwidth]{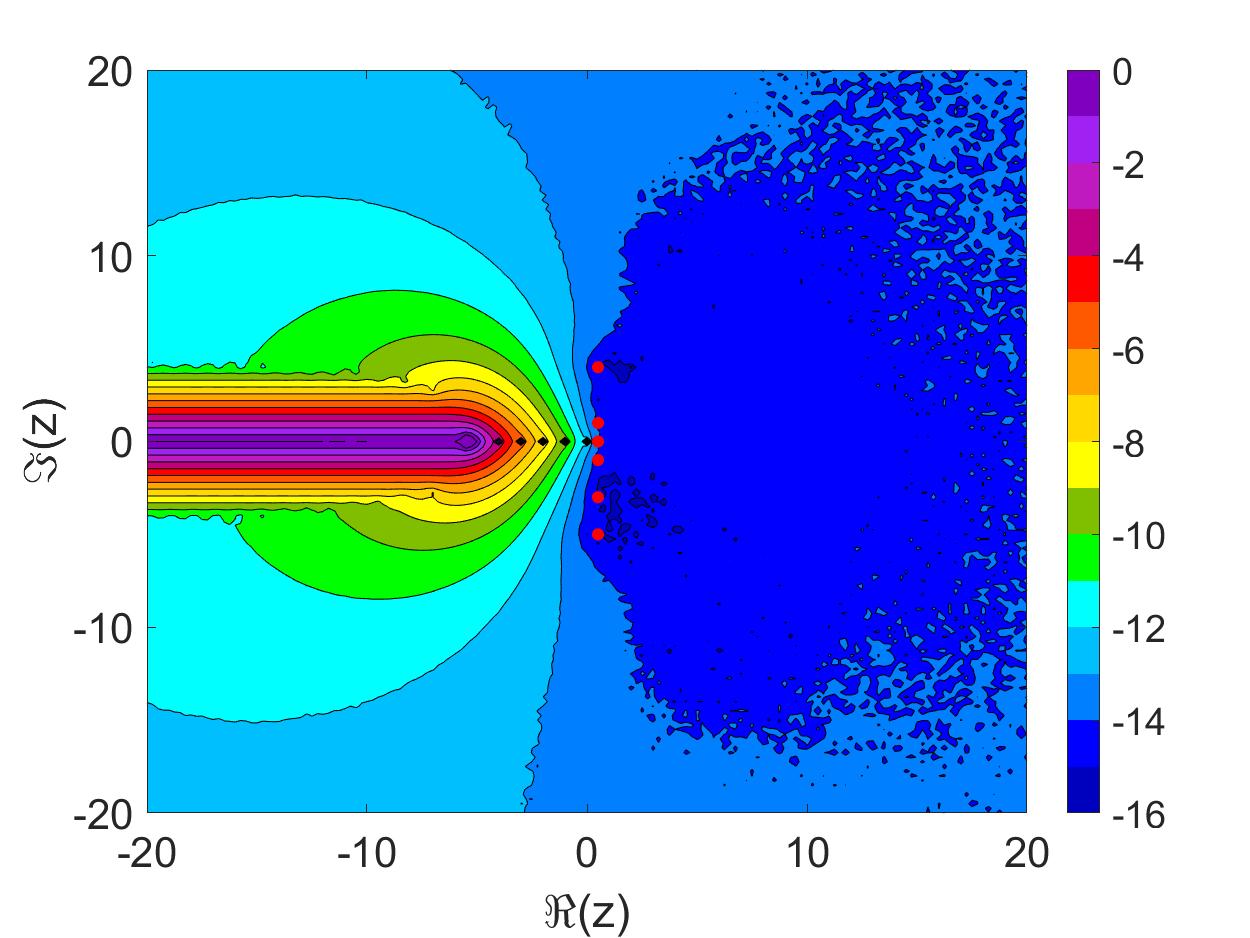} \hspace{5mm}
	\includegraphics[width=0.48\textwidth]{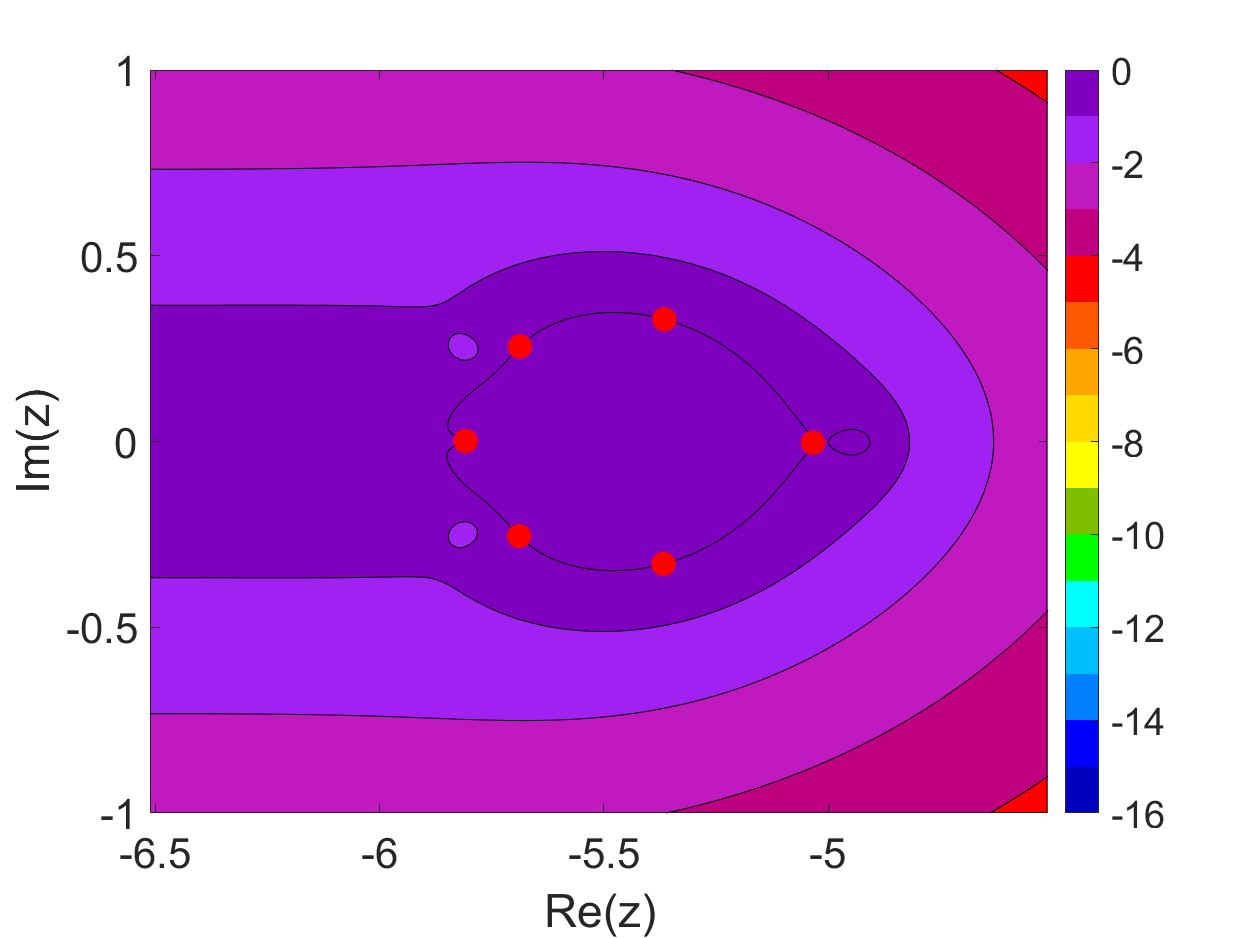}
	\caption{A filled contour plot of $\log_{10}(1-R(z)/\Gamma(z))$, with $R(z)$ a barycentric rational interpolant determined by The AAA algorithm \cite{AAA} applied to the gamma function (top left), and the scaled gamma function \eqref{eqn:Fr} (top right), with $r=5.51$ and 100 equally spaced test points from $[1,10]$. Results are best using 81 points equally spaced points along the line $[\frac{1}{2}-40i,\frac{1}{2}+40i]$ (bottom left, right). The zeros cluster around the branch point (bottom right).}
	\label{fig:AAA}
\end{figure}
We now turn our attention to rational approximation. Recently Nakatsukasa, S{\'e}te and Trefethen \cite{AAA} developed the AAA algorithm, which is quite remarkable in its own right, and is a great achievement in numerical analysis. As mentioned in our introductory remarks, a rational function $R(z)$ represented in barycentric form is adaptively constructed from $f(z)$, ensuring minimal degree up to some prescribed fixed precision $\epsilon$. The interpolation points are chosen adaptively from the test points using a greedy algorithm, and the error operates on the remaining test points, which in turn uses minimal function evaluation to retain algorithmic efficiency. In \cite{AAA}, the algorithm was applied directly to $\Gamma(z)$. The relative error for the rational function $R(z)$ as defined by the Matlab command \texttt{R = aaa(@gamma,linspace(1,10));} is shown in the top left panel of Figure \ref{fig:AAA}. Within an elliptical region containing the interval $[1,10]$, high precision is indeed achieved.

In light of our interpolation algorithm above, accuracy is limited by the fact that large $|z|$ asymptotics of $\Gamma(z)$ cannot be captured by a rational approximation. If we instead apply the AAA algorithm to the scaled gamma function \eqref{eqn:Fr}, the approximation is now of the form \eqref{eqn:R}, and improves dramatically (top right panel of Figure \ref{fig:AAA}). Indeed, with the same 100 test points from $[1, 10]$, nearly 10 decimals of precision are obtained over the right half-plane. Since there are 6 poles, we take $r=5.51$ (we found that the choice of $r$ did not need to be as specific), which produces a branch cut along the negative real axis emanating from $z=-5.51$.

Even better though (bottom left panel), is if we sample along the line of symmetry. With a mere 81 evenly spaced points along the segment $z = \frac{1}{2}+iy$, with $|y| \leq 40$, we produce  (lower left panel) more than 13 decimals of precision near the line of symmetry, and 14 over a large region of the right half plane! The gain in accuracy can almost certainly be attributed to the freedom allotted by varying the poles. Again, since the reflection formula would be utilized for the left half plane, the AAA algorithm ensures high precision over the whole complex plane.

The Matlab code \texttt{aaa.m} is part of Chebfun, and is available online. The Matlab function \verb|gamma_aaa| presented below in Table \ref{tab:AAA} uses the coefficients generated by the AAA algorithm. The zeros of $R(z)$, depicted in the lower right panel of Figure \ref{fig:AAA}, can be seen to cluster near the branch point at $z=-r$, but with a slightly different distribution from the preceeding constructions. See \cite{AAA}, and references therein, for more details.

One great advantage of \verb|gamma_aaa| is that the construction is obtained in double precision; numerical roundoff is minimized by making use of singular value decompositions on the tall Loewner matrix of size $81\times N$. Perhaps one drawback to the implementation is that the weights and function values are complex. However this is quite easily mitigated by demanding $\Im(\Gamma(x+0iy)) =0$, as shown in the code.
\begin{table}[ht]
	\caption{Matlab function for $\Gamma(z)$ yielding better than 13 digits of precision over the complex plane.}
	\label{tab:AAA}
\begin{lstlisting}
function g = gamma_aaa(zz)
% gamma function using a scaled barycentric rational interpolant. The numerical
% parameters can be reconstructed using
% >>[~,~,~,~,t,f,w] = aaa(@(z) gamma_B(z).*exp(z+5.51)./(z+5.51).^(z-1/2),1/2+1i*(-40:40),'tol',2*eps);
z	= zz(:);
r	= 5.51;
t	= 0.5 + 1i*[0 -1 1 40 -5 4 -3];
w	= [	-0.058033315398988594147056119254557;
		-0.12329392903700113481857414399201-0.05023735799303798155168720995789i;
		-0.072017314427899076223482666136988+0.029346047538194301729230772934898i;
		-0.73570545082472338371815112623153+0.35269523425582927078636430451297i;
		 0.39424018689617629229715589644911-0.046173606361601587932952384107921i;
		-0.10309397777341289259567247427185+0.04351009147705412610784847515788i;
		-0.17024770255373244953744915619609-0.32884604768510888872512509806256i];
f	= [ 722.24538019924227683077333495021;
		-47.561377245304413463600212708116+245.59392283177459148646448738873i;
		-47.561377245304413463600212708116-245.59392283177459148646448738873i;
		  2.3652595366167963319981026870664-1.1292734670349124925792239082512i;
		 -7.7668988926260489336073078447953+10.095560385519366519702089135535i;
		-14.060483019799770332269872596953-14.194015555290931729359726887196i;
		-27.239490936407644738892486202531+24.743535230939201596811471972615i];
I	= find(imag(z)==0  );						%ensure gamma(x) is real
J	= find(real(z)<=1/2);						%z for reflection formula
if ~isempty(J),	z(J)=1-z(J);	end				%reflect values
C	= 1./(z-t);									%Cauchy matrix
g	= (C*(w.*f))./(C*w);						%rational interpolant
for j=1:length(t),	g(z==t(j))= f(j);	end		%fix NAN at interp. pts.
g	= g.*exp((z-1/2).*log(z+r)-z-r);			
if ~isempty(I),		g(I) = real(g(I));	end		%make gamma(x) real
if ~isempty(J)									%apply reflection formula
	m	= fix(real(z(J)));						%avoid range reduction
	z(J)= z(J) - m;								
	g(J)= (-1).^m*pi./(g(J).*sin(pi*z(J)));		%reflect values 
end
g	= reshape(g,size(zz));
end
\end{lstlisting}
\end{table}

\subsection{Efficiency in double precision}
For a final comparison, we consider evaluation of $\Gamma(z)$ using the AAA-based algorithm, versus a shift-and-truncate approach to the Stirling series \eqref{eqn:SNK}. Overall algorithmic efficiency can be measured in various ways. We argue as follows.

If the goal is to evaluate $\Gamma(z)$ for many values over the complex plane, and maintain a uniform relative precision $\epsilon$, then \verb|gamma_aaa| is designed to be efficient in this way. Alternatively the Stirling series becomes increasingly accurate as $|z|$ increases. In order to maintain precision $\epsilon$ for $\Re(z) =\frac{1}{2}$, the increased accuracy mandated for larger $|z|$ will be destroyed, making it less efficient.

For comparison to the AAA-based algorithm, we implement the Matlab function \verb|gamma_B|, shown in Table \ref{tab:SNK}. In order to ensure $\approx 13$ decimals of relative precision over the right half-plane $\Re(z)\geq \frac{1}{2}$, we found that minimal choices of shifting $z\to z+16$ and truncating at $K = 5$ terms suffice.

Upon investigation of run times and accuracy, the codes are in fact quite similar. The AAA-based code tends to run 5-10\% faster. We conclude that this newly developed approach is competitive with the long-celebrated Stirling formula, in terms of overall efficiency.
\begin{table}[ht]
	\caption{Matlab function for $\Gamma(z)$ yielding 13 digits of precision over the complex plane.}
\label{tab:SNK}
\begin{lstlisting}
function g = gamma_B(z)
% gamma function using a shifted Stirling series
K	= 5;
N	= 16;
c	= log(2*pi)/2;
a	= [1/12, -1/360,  1/1260, -1/1680, 1/1188]; %a_n = B_2n/[(2n)(2n-1)]
J	= find(real(z)<=1/2);						%z for reflection formula
if ~isempty(J),	z(J)=1-z(J);	end				%reflect values
t	= z+N;
gam = a(1)./t;
for k = 2:K										%Horner evaluation of series
	t	= t.*(z+N).^2;
	gam	= gam + a(k)./t;
end
u	= 1;
for n = 0:N-1									%Pochhammer polynomial
	u	= u.*(z+n);
end
lg	= c -(z+N) +(z+N-1/2).*log(z+N) -log(u)+ gam;
g	= exp(lg);
if ~isempty(J)									%apply reflection formula
	m	= fix(real(z(J)));						%avoid range reduction
	z(J)= z(J) - m;								
	g(J)= (-1).^m*pi./(g(J).*sin(pi*z(J)));		
end
end
\end{lstlisting}
\end{table}
\section{Conclusion}
In this paper we have accomplished several tasks. First, we have shown that the methods developed by Lanczos and Spouge for approximating the gamma function can be recovered using techniques of interpolation. Next, we showed that by selecting different interpolating point distributions, the accuracy can be modestly improved, and as a result we obtain high degrees of nearly uniform precision for $\Gamma(z)$ over the right half of complex plane. Finally, we have shown that methods of rational approximation can also be used to do the same in double precision, if one is willing to allow the poles of the approximation to vary.

A general feature of our framework, as well as the earlier approaches of Lanczos and Spouge, is the favorable convergence properties with respect to the number $N$ of poles. The numerical results indicate a geometric convergence with respect to the degree $N$ of the rational approximation. We also note that the distribution of points at which $\Gamma(z)$ is sampled determines the region of interpolation. If this region is large enough, we observe that the extrapolation error cannot grow too large, and in particular the relative error $1-c_\infty(r)/\sqrt{2\pi}$ at infinite is well-controlled.

We have shown and discussed the errors obtained over the full complex plane. In practice, Euler's reflection formula \eqref{eqn:Ref} is applied, and so the actual error committed will be symmetric about the line $z = \frac{1}{2}+iy$, which then masks actual pole locations from the final approximation. Hence, we find that the overall best approach is to obtain a rational approximation $R(z) \approx F_r(z)$ to the scaled gamma function \eqref{eqn:Fr}, with interpolation points laid along the line of symmetry. Using the AAA algorithm, the construction is accomplished nearly instantly, and nearly full machine precision is obtained with a degree (6,6) rational function. The evaluation time is as good, and in fact slightly better than the Stirling series with truncation at $K=5$ and translation $z\to z+16$ to match the precision.

\section*{Acknowledgements}
The author is very grateful to an anonymous peer reviewer for several key suggestions that greatly improved the quality of this manuscript. The author is especially indebted for the suggestion to examine the AAA algorithm.

\bibliographystyle{amsplain}
\bibliography{Gamma_Bib}
	
\end{document}